\documentclass[12pt]{amsart}

\usepackage{graphicx}
\usepackage{amssymb}
\usepackage{amsmath, amscd}
\usepackage{mathrsfs}
\usepackage{tikz,tikz-cd,soul}

\usepackage{hyperref}

\DeclareFontFamily{U}{mathx}{\hyphenchar\font45}
\DeclareFontShape{U}{mathx}{m}{n}{
      <5> <6> <7> <8> <9> <10>
      <10.95> <12> <14.4> <17.28> <20.74> <24.88>
      mathx10
      }{}
\DeclareSymbolFont{mathx}{U}{mathx}{m}{n}
\DeclareFontSubstitution{U}{mathx}{m}{n}
\DeclareMathAccent{\widecheck}{0}{mathx}{"71}
\DeclareMathAccent{\wideparen}{0}{mathx}{"75}

\numberwithin{equation}{section}

\newtheorem{theorem}{Theorem}[section]
\newtheorem{proposition}[theorem]{Proposition}
\newtheorem{corollary}[theorem]{Corollary}
\newtheorem{lemma}[theorem]{Lemma}

\theoremstyle{definition}
\newtheorem{definition}[theorem]{Definition}
\newtheorem{example}[theorem]{Example}

\theoremstyle{remark}
\newtheorem{remark}[theorem]{Remark}

\numberwithin{equation}{section}

\newcommand{\R}{{\mathbb{R}}}
\newcommand{\C}{{\mathbb{C}}}

\usepackage[margin=1in,marginparwidth=0.8in, marginparsep=0.1in]{geometry}

\begin{document}

\title{Chekanov--Eliashberg dg-algebras for singular Legendrians}
\author{Johan Asplund}
\address{Department of mathematics, Uppsala University, Box 480, 751 06 Uppsala, Sweden}
\email{johan.asplund@math.uu.se}
\author{Tobias Ekholm}
\address{Department of mathematics, Uppsala University, Box 480, 751 06 Uppsala, Sweden \and
Institut Mittag-Leffler, Aurav 17, 182 60 Djursholm, Sweden}
\email{tobias.ekholm@math.uu.se}

\thanks{JA is supported by the Knut and Alice Wallenberg Foundation.}
\thanks{TE is supported by the Knut and Alice Wallenberg Foundation and the Swedish Research Council.}

\begin{abstract}
The Chekanov--Eliashberg dg-algebra is a holomorphic curve invariant associated to Legendrian submanifolds of a contact manifold. We extend the definition to Legendrian embeddings of skeleta of Weinstein manifolds. Via Legendrian surgery, the new definition gives direct proofs of wrapped Floer cohomology push-out diagrams \cite{GPS}. It also leads to a proof of a conjectured isomorphism \cite{EL,sylvan} between partially wrapped Floer cohomology and Chekanov--Eliashberg dg-algebras with coefficients in chains on the based loop space.
\end{abstract}

\maketitle

\section{Introduction}\label{sec:intr}
We introduce holomorphic curve invariants of Legendrian embeddings of a Weinstein domain $V$ of dimension $(2n-2)$ with a given handle decomposition $h$ into the $(2n-1)$-dimensional contact boundary $\partial W$ of a Weinstein $2n$-manifold $W$. Here, a Legendrian embedding $f\colon V\to\partial W$ is an embedding that extends to a contact embedding $F\colon(-\epsilon,\epsilon)\times V\to\partial W$, $f=F|_{V\times\{0\}}$, where $\R\times V$ is the contactization of $V$, for some $\epsilon>0$. The image of a Legendrian embedding is often called a Weinstein hypersurface, see e.g., \cite{sylvan,GPS1}.

For such an embedding the core disks of the handles in $h$ are isotropic disks in $\partial W$. In particular, the $(n-1)$-dimensional core disks $l$ are Legendrian and we think of the whole skeleton of $V$, i.e., the union $l_{0}\cup l$ of subcritical core disks $l_0$ and $l$, as a singular Legendrian in $\partial W$. 

To consider ordinary smooth Legendrian $(n-1)$-submanifolds $\Lambda\subset \partial W$ from this point of view, take $V$ to be a small neighborhood of the zero-section in the cotangent bundle $T^{\ast}\Lambda$ and equip it with a handle decomposition $h$ with a single top handle. Then the invariant of $(V,h)$ agrees with the Chekanov--Eliashberg dg-algebra of the Legendrian submanifold $\Lambda$ with coefficients in chains on the based loop space, see Theorem \ref{t:loopspacecoeff}. 

The generalization of dg-algebras of Legendrian submanifolds to embeddings of general Weinstein domains $V\subset\partial W$ is useful from many points of view. For example, it gives Legendrian surgery formulas for wrapped Floer cohomology in $W$ stopped at $V$, see Theorem \ref{t:basic}, dg-algebras for non-closed Legendrian submanifolds $\Lambda\subset\partial X$ with Legendrian boundary in $\partial V$, and cut-and-paste formulas for holomorphic curve invariants that parallels results in \cite{GPS}, see Remark \ref{r:GPS}.  

The remainder of Section \ref{sec:intr} is organized as follows. In Section \ref{ssec:cotangentS^n} we describe our construction in the basic example of the cotangent bundle $T^{\ast}\R^{n}$ of $n$-space. In Section \ref{ssec:main} we define Chekanov--Eliashberg dg-algebras in general and state our main results. In Section \ref{ssec:gencompex} we discuss generalizations and computations.

\subsection{A basic example}\label{ssec:cotangentS^n}
We give a Legendrian surgery description of $T^{\ast}\R^{n}$ and assume for simplicity that $n>2$ so that $\R^{n}\setminus \{0\}$ is simply connected. (We consider $n=2$ in Example \ref{ex:unknot} following \cite{EL2}.) From the Weinstein handle perspective, the cotangent bundle of the sphere $T^\ast S^{n}$ is more basic than $T^{\ast}\R^{n}$ and we start from there.

Consider the standard symplectic $2n$-space $B=(\R^{2n}, \sum_{j=1}^{n}dx_{j}\wedge dy_{j})$ with ideal contact boundary the standard contact sphere $\partial B=S^{2n-1}$. Let $\Lambda_{n-1}\subset S^{2n-1}$ denote the Legendrian unknot, a Legendrian $(n-1)$-sphere Legendrian isotopic to the ideal boundary of a Lagrangian $n$-plane through the origin in $\R^{2n}$. The Chekanov--Eliashberg dg-algebra of $\Lambda_{n-1}$, $CE^{\ast}(\Lambda_{n-1};B)$ is the unital algebra $\C[a_{n-1}]$ on one generator $a_{n-1}$, corresponding to the unique Reeb chord of the standard representative of  $\Lambda_{n-1}$ in a Darboux ball (Figure \ref{fig:point_constraint} left), of degree $|a_{n-1}|=-(n-1)$, see \cite[Section 7.1]{BEE}. Here, unlike in \cite{BEE} but in line with \cite{EL}, we use cohomological grading: Reeb chords are graded by a shift of the negative of their Conley--Zehnder index and the differential increases degree by $1$. 

Attaching a Weinstein $n$-handle to $B$ (viewed as the $2n$-ball with a half-infinite symplectization collar attached) along $\Lambda_{n-1}$ using the standard attaching map, the resulting Weinstein manifold is the cotangent bundle $T^{\ast}S^{n}$ and the co-core disk $C$ in the handle becomes the cotangent fiber in $T^{\ast}S^{n}$. Legendrian surgery \cite{BEE,EkholmHol,EL}, see Remark \ref{r:overviewLegsurg} for an outline, gives a geometrically induced quasi-isomorphism of $A_{\infty}$-algebras from the wrapped Floer cohomology of $C$, $CW^{\ast}(C;T^{\ast}S^{n})$, to $CE^\ast(\Lambda_{n-1};B)$, where $CE^{\ast}(\Lambda_{n-1};B)$ is viewed as a chain complex generated by monomials in Reeb chords with differential induced by the dg-algebra differential, with product given by concatenation of monomials, and with all higher $A_{\infty}$-operations trivial. In this simple example, $CE^{\ast}(\Lambda_{n-1};B)$ can be directly compared to the description of $CW^{\ast}(C;T^{\ast}S^{n})$ as chains on the based loop space $C_{-\ast}(\Omega S^{n})$ of $S^{n}$ with the Pontryagin product, see \cite{AbouzaidBased}. Concretely, we use Adams' construction \cite{Adams,EL} to represent $C_{-\ast}(\Omega S^{n})$ as the reduced cobar construction on the Morse complex of $S^{n}$. Taking a Morse function with only two critical points then gives a quasi-isomorphism of algebras $C_{-\ast}(\Omega S^{n})\approx \C[y_{n-1}]$ where the generator $y_{n-1}$ has degree $-(n-1)$ and corresponds to the maximum of the Morse function. The surgery isomorphism then maps $y_{n-1}$ to $a_{n-1}$. 

Next, consider $T^{\ast}\R^{n}$. Here we view $\Lambda_{n-1}$ instead as a Legendrian stop \cite{sylvan,GPS} in the ideal contact boundary of $B$. This can be thought of, see \cite{ENS,EL,AsplundFiber} as attaching a punctured version of the Weinstein handle, $T^{\ast}(\Lambda_{n-1}\times[0,\infty))$, to $\Lambda_{n-1}$. In the case under consideration, it is clear that the result is $T^{\ast}\R^{n}$. 
According to \cite[Conjecture 3]{EL} the wrapped Floer cohomology $CW^{\ast}(C;T^{\ast}\R^{n})$ is isomorphic to $CE^\ast(\Lambda_{n-1},C_{-\ast}(\Omega \Lambda_{n-1});B)$, the Chekanov--Eliashberg dg-algebra with loop space coefficients. Concretely (recall $n>2$), this dg-algebra is a free algebra on two generators $\C\langle a_{n-1},y_{n-2}\rangle$, with differential $\partial a_{n-1}=y_{n-2}$ induced by a count of rigid disks with a point constraint, see Figure~\ref{fig:point_constraint}. It is then quasi-isomorphic to the ground field $\C$ in agreement with the Floer cohomology of the fiber in $T^{\ast}\R^{n}$. 

We now instead consider $T^{\ast}\R^{n}$ from the point of view in the current paper, where we take a Weinstein handle perspective and represent $T^{\ast}\R^{n}$ as a Weinstein cobordism with negative end the contactization of $T^{\ast}\Lambda_{n-1}$ and one critical handle with co-core disk $C$ at the critical point, $C$ then corresponds to the fiber in $T^{\ast}\R^{n}$. More precisely, let $V\subset \partial B$ be a small neighborhood of the zero section in $T^{\ast}\Lambda_{n-1}$, equip it with a handle decomposition $h$ with a zero handle with core disk $l_{0}$ and an $(n-1)$-handle with core disk $l$. Let $V_{0}$ denote the subcritical part of $V$, i.e., a neighborhood of $l_{0}$. We construct a handle decomposition on the Weinstein cobordism representing $T^{\ast}\R^{n}$ as follows. Start from the $2n$-ball and first attach the handle $V_{0}\times D^{\ast}_{\epsilon}[-1,1]$, where $D^{\ast}_\epsilon[-1,1]$ denotes an $\epsilon$-disk subbundle of the cotangent bundle $T^{\ast}[-1,1]$, along $(-\epsilon,\epsilon)\times V_{0}$ to $\partial B\sqcup (\R\times T^{\ast}\Lambda_{n-1})$ viewed as the positive end of $B\sqcup (\R\times(\R\times T^{\ast}\Lambda_{n-1}))$. Denote the resulting manifold $B_{V}^{0}$.

To connect to chains on the based loop space, take $l_{0}$ to be the constraining point in $\Lambda_{n-1}$ corresponding to the generator $y_{n-1}$ (the maximum of the Morse function). Then the boundary $\partial l$ of the core disk of the top handle $l$ intersects the boundary $\partial V_{0}$ of the neighborhood in a standard Legendrian $(n-2)$-sphere $\Lambda_{n-2}$. The full cobordism corresponding to $T^{\ast}\R^{n}$ is obtained by attaching a standard Weinstein handle to the Legendrian $(n-1)$-sphere $\Sigma(h)$ which consists of two copies of $l$ joined across the handle via $\partial l\times [-1,1]= \Lambda_{n-2}\times [-1,1]$. 

This attaching sphere $\Sigma(h)$ then has two Reeb chords: $a_{n-1}$ and $a_{n-2}$, the latter inside the handle. The differential in $CE^{\ast}(\Sigma(h);B_{V}^{0})$ satisfies $\partial a_{n-1}=a_{n-2}$, see Figure~\ref{fig:point_constraint} and $CE^{\ast}(\Sigma(h);B_{V}^{0})$ is isomorphic to $CW^{\ast}(C;T^{\ast}\R^{n})$. Our approach here is to \emph{define} the Chekanov--Eliashberg dg-algebra of $(V,h)$, $CE^{\ast}((V,h);B)$ as $CE^{\ast}(\Sigma(h);B_{V}^{0})$, the usual dg-algebra of the Legendrian sphere $\Sigma(h)$. For smooth Legendrians ($V$ a small neighborhood of the zero section $T^{\ast}\Lambda$) the new definition will correspond to the usual Chekanov--Eliashberg dg-algebra with loop space coefficients, see Theorem \ref{t:loopspacecoeff}, but the definition makes sense more generally for any Legendrian admitting a Weinstein thickening, see \cite{AB} for one-dimensional examples.  

If we compare the usual chains on the loop space dg-algebra with the one defined here, we find that in the former approach the algebra contains generators of both topological (chains on the based loop space) and dynamical (Reeb chords) nature, whereas in the latter all generators are of dynamical nature.

\begin{figure}[!htb]
	\centering
	\includegraphics{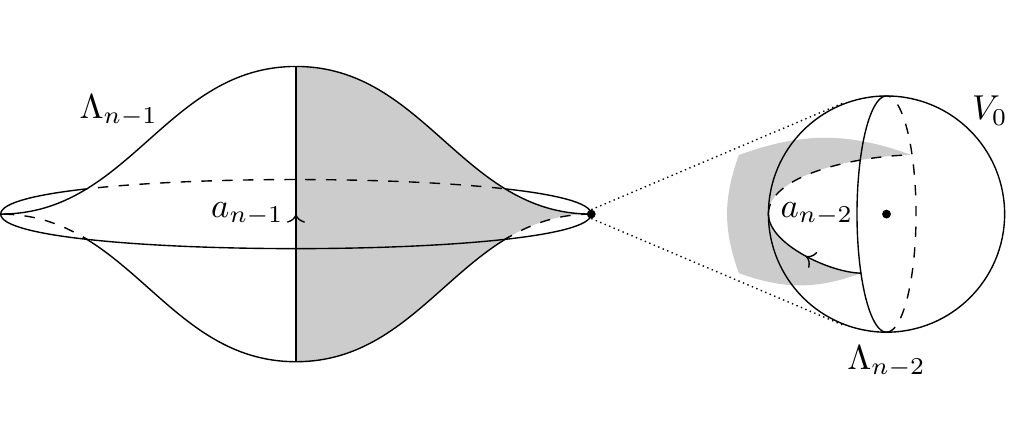}
	\label{fig:point_constraint}
	\caption{Front projection of Legendrian unknot $\Lambda_{n-1}$ with a point constrained holomorphic curve that gives $\partial a_{n-1} = a_{n-2}$.}
\end{figure}

\subsection{General definition and main results}\label{ssec:main}
Our definition of dg-algebras for general Weinstein domains is the following generalization of the cobordism method described for $T^{\ast}\R^{n}$ in Section \ref{ssec:cotangentS^n}. Consider a Weinstein $(2n-2)$-domain $V$ with handle decomposition $h$, where $h$ encodes both handles and attaching maps. We write $V_{0}$ for the subcritical part of $V$, $\partial V_{0}$ for its contact boundary, $l_{j}$, $j=1,\dots,m$ for the core-disks of the critical handles and $\partial l_{j}$ for their boundaries, i.e., the attaching spheres for the corresponding handle, and $\partial l=\bigcup_{j=1}^{m}\partial l_{j}$ for their union. Assume that $(V,h)\subset \partial W$ is a Legendrian embedding into the contact boundary of a Weinstein $2n$-manifold $W$. For simplicity, we will assume throughout that $c_{1}(W)=c_{1}(V)=0$. 

We build a cobordism $W_{V}$ (which is our version of $W$ stopped at $V$, see \cite{sylvan, GPS}, compare also \cite{avdek}) with negative end $\R\times V$ (in the negative end we think of $V$ as the Weinstein manifold which is the completion of the embedded Weinstein domain) in two steps: First we construct $W_{V}^{0}$ by attaching $V_{0}\times D^{\ast}_{\epsilon}[-1,1]$, where $D_{\epsilon}^{\ast}[-1,1]$ denotes an $\epsilon$-disk subbundle of the cotangent bundle $T^{\ast}[-1,1]$, to $\partial W\sqcup (\R\times V)$ along $V_{0}\times D^{\ast}_{\epsilon}[-1,1]|_{\{-1,1\}}$. Then the full cobordism $W_{V}$ is obtained from $W_{V}^{0}$ by attaching critical Weinstein $n$-handles to a link of Legendrian spheres $\Sigma(h)\subset\partial W_{V}^{0}$ defined as follows. There is one sphere component $\Sigma(h_{j})\approx S^{n-1}$ of $\Sigma(h)$ for each top-dimensional handle $h_{j}$ of $V$. Here $\Sigma(h_{j})$ consists of two copies of $l_{j}\approx D^{n-1}$, one embedded in $\partial W\setminus ((-\epsilon,\epsilon)\times V_{0})$ with boundary on $\{0\}\times \partial V_{0}$ and one in $(\R\times V)\setminus((-\epsilon,\epsilon)\times V_{0})$. The copies are joined across the handle $V_{0}\times D^{\ast}_{\epsilon}[-1,1]$ by $\partial l_{j}\times [-1,1]\approx S^{n-2}\times[-1,1]$, see Figure \ref{fig:W_stopped_subcrit}.

\begin{figure}[!htb]
	\includegraphics{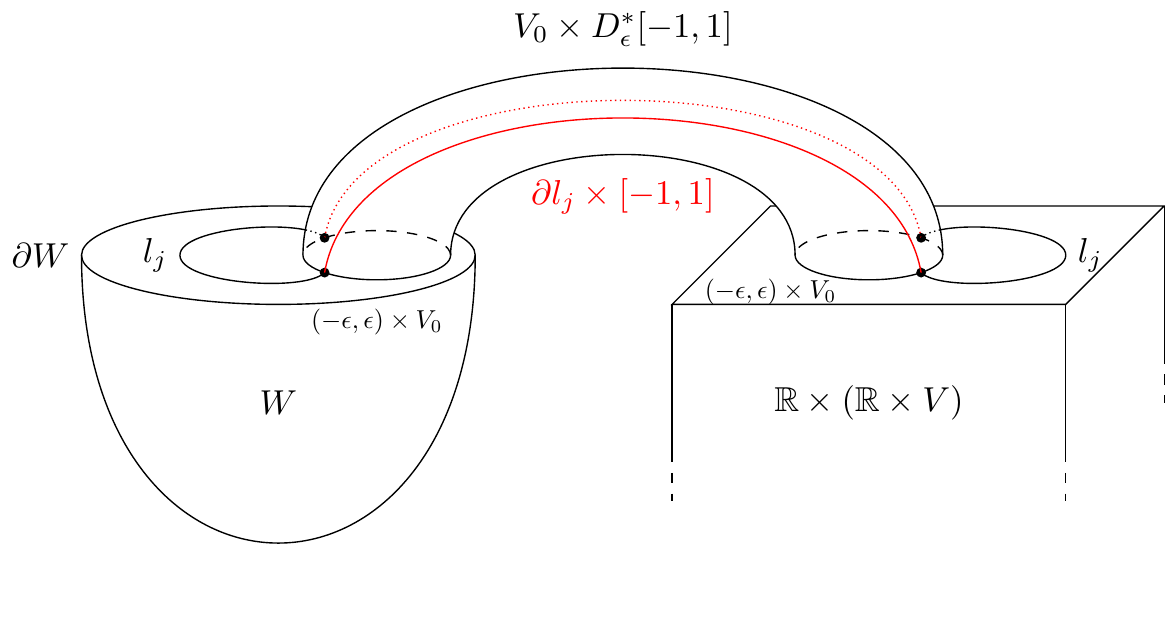}
	\caption{The Weinstein cobordism $W^0_V$ is constructed by attaching $V_0 \times D^\ast_{\epsilon}[-1,1]$ along $V_{0}\times D^{\ast}_{\epsilon}[-1,1]|_{\{-1,1\}}$.}
	\label{fig:W_stopped_subcrit}
\end{figure}
 
We then \emph{define} the dg-algebra of $(V,h)$, $CE^{\ast}((V,h);W)$ as the usual dg-algebra of its attaching spheres, 
\begin{equation}\label{eq:defCE(V)} 
CE^{\ast}((V,h);W) \ := \ CE^{\ast}(\Sigma(h);W_{V}^{0}),
\end{equation}
where $\Sigma(h)\subset\partial W_{V}^{0}$ is the Legendrian link of attaching spheres described above and where $W_{V}^{0}$ in $CE^{\ast}(\Sigma(h);W_{V}^{0})$ indicates that the differential counts holomorphic disks anchored in $W_{V}^{0}$, see \cite{BEE,EkholmHol,EL}. It is easy to see that Legendrian isotopies of $(V,h)$ induce Legendrian isotopies of $\Sigma(h)$ and it follows that $CE^{\ast}((V,h);W)$ is a Legendrian isotopy invariant. 

The dg-algebra in \eqref{eq:defCE(V)} naturally contains the dg-algebra of the attaching spheres of $V$ as a subalgebra. Indeed, as explained above, the attaching spheres $\partial l\subset \partial V_{0}$ of the top handles of $V$ in $\partial V_{0}$ appears as `equators' of the spheres in $\Sigma(h)$ and, shrinking the handle, a straightforward action argument shows that equipping $V_{0}$ with what we call an index definite Weinstein structure, see Section \ref{sec:Weinhand}, the chords of $\partial l$ generate a dg-subalgebra of $CE^{\ast}(\Sigma(h);W_{V}^{0})$ canonically identified with $CE^{\ast}(\partial l;V_{0})$.

It follows by Legendrian surgery \cite{BEE,EkholmHol,EL} that the dg-algebra $CE^{\ast}((V,h);W)$ contains the wrapped Floer cohomology algebra $CW^{\ast}(c;V)\approx CE^{\ast}(\partial l;V_{0})$ of the union $c=\bigcup_{j} c_{j}$ of the co-core $(n-1)$-disks dual to $l_{j}$ in the critical handles $h_{j}$ of $V$, as a subalgebra. From the point of view of Floer cohomology with coefficients it is then clear that our dg-algebras have the most general coefficients: the endomorphism algebra of generators of the Fukaya category of the embedded domain itself. In this sense our construction generalizes Floer cohomology with coefficients in chains on the based loop space to objects in the Fukaya category represented by skeleta of Weinstein manifolds, no matter how singular their skeleta may be, see Section \ref{sssec:Floercohomology}. 

With $CE^{\ast}((V,h);W)$ defined we consider its properties. Let $W_{V}$ denote the Weinstein manifold $W$ stopped at $V$, see \cite{EL,sylvan,GPS}. Let $C(h)=\bigcup_{j=1}^{m}C(h_{j})$ denote the union of Lagrangian co-core $n$-disks $C(h_{j})$ of $W_{V}$ in the Weinstein $n$-handles attached to $\Sigma(h_{j})\subset \partial W_{V}^{0}$. 

\begin{theorem}\label{t:basic}
There is a natural surgery isomorphism
\[ 
\Phi\colon CW^{\ast}(C(h);W_{V})\longrightarrow CE^{\ast}((V,h);W)
\]
of $A_{\infty}$-algebras, where the right hand side is viewed as a complex generated by Reeb chord monomials with differential induced by the dg-algebra differential, product given by concatenation, and all higher operations trivial. 
\end{theorem}
Given the Legendrian surgery isomorphism from \cite{BEE,EkholmHol,EL}, Theorem \ref{t:basic} follows immediately from the  definition of $CE^\ast((V,h);W)$ and is one of its main motivations.

We next relate our new definition of the dg-algebra to the standard version. Let $\Lambda\subset \partial W$ be a smooth Legendrian and consider its Chekanov--Eliashberg dg-algebra with coefficients in chains on the based loop space $CE^{\ast}(\Lambda,C_{-\ast}(\Omega\Lambda);W)$, as defined in \cite{EL}. Let $(V(\Lambda),h(\Lambda))$ denote a small disk sub-bundle of the cotangent bundle $T^{\ast}\Lambda$ of $\Lambda$ with a handle decomposition with a single top-dimensional handle.
\begin{theorem}\label{t:loopspacecoeff}
	There is a natural quasi-isomorphism
	\[ 
	\Psi\colon CE^{\ast}((V(\Lambda),h(\Lambda));W)\longrightarrow CE^{\ast}(\Lambda,C_{-\ast}(\Omega\Lambda);W).
	\]
\end{theorem}
Theorem \ref{t:loopspacecoeff} is proved in Section \ref{sec:newandold}.
We point out that Theorems \ref{t:basic} and \ref{t:loopspacecoeff} together establish \cite[Conjecture 3]{EL}.

Consider next two Weinstein manifolds $W$ and $W'$ that both contain $(V,h)$ as a Legendrian embedding in their boundary. Then we can join $W$ and $W'$ by attaching a $V$-handle, $V\times D^{\ast}_{\epsilon} [-1,1]$, along $(-\epsilon,\epsilon)\times V$ to build a new Weinstein manifold $W\#_{V} W'$. As in the construction of the handle decomposition of $X$ above, the $V$-handle gives a link of Legendrian attaching spheres $\Sigma_{\#}(h)$ in $\partial(W\#_{V_{0}}W')$, where $V_{0}$ is the subcritical part of $V$ and where a component $\Sigma_{\#}(h_{j})$ consists of two copies of $l_{j}$, inside $\partial W$ and $\partial W'$, joined by $\partial l_{j}\times [-1,1]$ across the handle.  
\begin{theorem}\label{t:vankampen}
There is a natural push-out diagram
\[ 
\begin{CD}
CE^{\ast}(\partial l;V_{0}) @>>> CE^{\ast}((V,h);W') \\
@VVV @VVV \\
CE^{\ast}((V,h);W) @>>> CE^{\ast}(\Sigma_{\#}(h);W\#_{V_{0}} W')
\end{CD},
\]
where, for a suitable contact form and almost complex structure on the $V_{0}$-handle, all maps are defined by inclusion on Reeb chord generators.
\end{theorem}   

\begin{remark}\label{r:GPS}
Combining Theorems \ref{t:basic} and \ref{t:vankampen} and the Legendrian surgery isomorphism $CW^{\ast}(C_{\#}(h);W\#_{V} W')\approx CE^{\ast}(\Sigma_{\#}(h);W\#_{V_{0}} W')$, where $C_{\#}(h_{j})$ is the co-core disk in the Weinstein handle attached to $\Sigma_{\#}(h_{j})$ and $C_{\#}(h)=\bigcup_{j} C_{\#}(h_{j})$, we obtain the following push-out diagram of \cite{GPS}:
\[ 
\begin{CD}
	CW^{\ast}(c;V) @>>> CW^{\ast}(C'(h);W'_{V}) \\
	@VVV @VVV \\
	CW^{\ast}(C(h);W_{V}) @>>> CW^{\ast}(C_{\#}(h);W\#_{V} W')
\end{CD},
\]
where $W_{V}$ and $W'_{V}$ denote $W$ and $W'$ stopped at $V$ and $C(h)$ and $C'(h)$ denote the union of Lagrangian disks dual to the top handles $h$ of $V$ in $W$ and $W'$, respectively.
\end{remark}

\begin{remark}
Our definition of the Chekanov--Eliashberg dg-algebra of a Legendrian embedding of a Weinstein domain $V$ involves a handle decomposition $h$ of $V$ and $CE^{\ast}((V,h);W)$ indeed depends on $h$, see Example \ref{ex:noninv} for an illustration in a simple example. Stronger forms of invariance of the dg-algebra requires rather restrictive deformations: $CE^{\ast}((V,h);W)$ is defined in terms of ordinary Chekanov--Eliashberg dg-algebras of attaching spheres, for $V$ and for the manifolds that result by attaching a $V$-handle to $W$. As explained above, the quasi-isomorphism class of $CE^{\ast}((V,h);W)$ is independent up to Legendrian isotopies of the handles in $h$, with compatible isotopies of the attaching maps. 

Using the isomorphism in Theorem \ref{t:basic} and known properties of Fukaya categories one can find characteristics of $CE^{\ast}((V,h);W)$ that are independdent of the handle decomposition $h$. Since $CE^{\ast}((V,h);W)$ is isomorphic to the endomorphism algebra of a set of co-core disks that together with other co-core disks already present in $W$ generate the Fukaya category of $W_{V}$. There is a cobordism map taking the wrapped Fukaya category of $W_{V}$ to that of $W$ and its kernel is generated by the co-core disks of the stop, see \cite[Section 1.2]{sylvan2} and \cite{GPS}. It then follows that their derived module categories are independent of the generating set and hence of $h$. As an example of a consequnece of this, using properties of Fukaya categories, see e.g.~\cite{GPS1}, or alternatively directly using the Legendrian surgery description of symplectic homology \cite{BEE}, it follows that the Hochschild homology of $CE^{\ast}((V,h);W)$ is isomorphic to the mapping cone of the cobordism map on symplectic cohomology induced by the Weinstein cobordism between $W_{V}$ and $W$, and is hence independent of $h$.   
\end{remark}

\subsection{Generalizations, computations, and examples}\label{ssec:gencompex}
We also describe natural generalizations of the basic dg-algebras introduced above and illustrate how to work with them by describing schemes for computation and studying examples. 

In Section \ref{sec:Legwbdry} we discuss relative versions of the dg-algebra. More precisely, with $V\subset \partial W$ as above if $\Lambda\subset \partial W$ is a Legendrian submanifold with boundary $\partial\Lambda$ in $\partial V$ then there is a Chekanov--Eliashberg dg-algebra $CE^\ast(\Lambda;V;W)$ associated to $\Lambda$. We show that relative versions of the theorems above hold. In Section \ref{sec:cobordism} we discuss analogues of exact Lagrangian fillings and cobordisms for our singular Legendrians $(V,h)$ and sketch constructions of their Floer cohomology and of induced cobordism maps.   

In Section \ref{sec:exandappl} we describe how to find generators and compute differentials in the case when the contact manifold $\partial W$ is a contactization. We also study several concrete examples, e.g., we exhibit Legendrian tri-valent graphs that do not admit Lagrangian fillings with restricted singularities.  

\subsection*{Acknowledgements} TE thanks Sylvain Courte and Vivek Shende for valuable discussions. The authors thank Baptiste Chantraine and an anonymous referee for careful reading and valuable suggestions for improvements of the text.

\section{Weinstein handles and decompositions of $V$-handles}\label{sec:handle_decomp_V}
In this section we discuss handle decompositions of the basic building block of our construction. Let $(V,h)$ be a $(2n-2)$-dimensional Weinstein manifold with Liouville form $\lambda$, Liouville vector field $z$, handle decomposition $h$, and exhausting Morse function $\phi\colon V\to [0,\infty)$ with a single non-degenerate minimum at value $0$. Then there is $\rho>0$ such that $\phi^{-1}([\rho,\infty))\approx [0,\infty) \times \partial V$ is the positive half of the symplectization of the contact boundary $\partial V$ of $V$. By scaling, we may take $\rho>0$ arbitrarily small. Let $F^{z}(v,t)$ denote the flow of $z$ with initial condition $v\in V$. The \emph{skeleton} of $V$ is the set of points $v\in V$ such that $\lim_{t\to\infty}\phi(F^{z}(v,t))<\infty$. The skeleton is then a union of core disks in the handles of the decomposition $h$. We will often think of the exhausting function $\phi$ as being approximately zero on the skeleton so that $\phi^{-1}(\delta)$ for $\delta>0$ gives a decreasing family of neighborhoods of the skeleton as $\delta\to 0$.

We next discuss local models for our handles. We start with basic Weinstein handles and index definite Weinstein structures on subcritical Weinstein manifolds and then go to $V$-handles.

\subsection{Weinstein handles}\label{sec:Weinhand}
Consider $\R^{2n}$ with the standard symplectic form $\omega=\sum_{j=1}^{n} dx_{j}\wedge dy_{j}$. For $0\le k\le n$, the vector field 
\[ 
Z_{k}=\sum_{j=1}^{k}(2x_{j}\partial_{x_{j}}-y_{j}\partial_{y_{j}}) + \sum_{j=k+1}^{n}\frac12(x_{j}\partial_{x_{j}}+y_{j}\partial_{y_{j}})
\]
is a Liouville vector field, $L_{Z_{k}}\omega=\omega$. We consider for $\delta\ge 0$ the contact hypersurfaces 
\[ 
\partial_{\pm\delta}H_{k}=\left\{(x,y)\colon \sum_{j=1}^{k}(x_{j}^{2}-\tfrac12 y_{j}^{2})+\sum_{j=k+1}^{n}\tfrac14(x_{j}^{2}+y_{j}^{2})=\pm\delta\right\}
\]
and $\R^{2n}=H_{k}$ with the Liouville field $Z_{k}$ as a symplectic cobordism with negative and positive ends 
\begin{align*}
	\partial_{-\delta} H_{k}&\approx S^{k-1}\times D^{k+2(n-k)},\\
	\partial_{\delta} H_{k}&\approx D^{k}\times S^{2n-k-1}.
\end{align*}
We say that $D^{k}\times \{0\}\subset H_{k}$ is the \emph{core disk} of the handle with \emph{attaching sphere} $S^{k-1}\times \{0\}\subset \partial_{-\delta}H_{k}$, and that $\{0\}\times D^{2n-k}\subset H_{k}$ is the \emph{co-core disk} and $\{0\}\times S^{2n-k-1}\subset \partial_{\delta} H_{k}$ the \emph{co-core sphere}.

The induced contact form on $\partial_{\pm\delta} H_k$ is the restriction of the form
\[ 
\alpha_{k}=\sum_{j=1}^{k}(2x_{j}dy_{j}+y_{j}dx_{j})+\sum_{j=k+1}^{n}\frac 12(x_{j}dy_{j}-y_{j}dx_{j}).
\] 
The corresponding Reeb vector field $R_{k;\pm\delta}$ along $\partial_{\pm \delta} H_{k}$ is
\[ 
R_{k;\pm\delta}= \frac{1}{N(x,y)}\left(\sum_{j=1}^{k}(2x_{j}\partial_{y_{j}}+y_{j}\partial_{x_{j}})+\sum_{j=k+1}^{n} \frac 12(x_{j}\partial_{y_{j}}-y_{j}\partial_{x_{j}})\right),
\]
where
\[ 
N(x,y)=\sum_{j=1}^{k}(4x_{j}^{2}+y_{j}^{2})+\sum_{j=k+1}^{n}\frac 14(x_{j}^{2}+y_{j}^{2}).
\] 
It follows in particular that any Reeb orbit in $\partial_{\delta} H_{k}$ lies in the middle of the handle, i.e., at $(x_{j},y_{j})=(0,0)$ for $j=1,\dots,k$, where the Reeb flow agrees with the standard Reeb flow on the $(2(n-k)-1)$-dimensional sphere.

Along $\partial_\delta H_k$, we pick an almost complex structure $J_0$ which takes $Z_{k}$ to $R_{k}$. We extend $J_0$ as an almost complex structure $J$ over $H_{k}$ that leaves the contact planes in $\partial_{\pm\delta} H_{k}$ invariant for all $\delta>0$ and agrees with a Liouville invariant almost complex structure on the symplectization ends built on $\partial_{\pm\delta}H_{k}$ for any $\delta>0$. We call such an almost complex structure \emph{handle adapted}.

An arbitrary Weinstein $2n$-manifold $W$ admits a handle decomposition $H$, which can be described as follows in terms of an exhausting function $\phi\colon W\to [0,\infty)$. The function $\phi$ has $d$, $1\le d\le n$ critical levels $0< \delta_{1}<\dots<\delta_{d}<\infty$. For $0<\delta<\delta_{1}$, $\phi^{-1}([0,\delta))$ is isomorphic to the $2n$-ball, $H_{0}$ in the notation above. Furthermore, for $\delta_{k-1}<\delta'<\delta_{k}<\delta<\delta_{k+1}$, $\phi^{-1}([0,\delta])$ is obtained from $\phi^{-1}([0,\delta'))$ by attaching a finite number of $k$-handles $H_{k}$ along an isotropic link $\partial L_{k}$ of $(k-1)$-dimensional isotropic attaching spheres in its contact boundary, $\phi^{-1}(\delta')$. If $d<n$ then we say that $W$ is subcritical. 

Our next result controls Reeb orbits in the boundary of a subcritical manifold. We consider a parameter $\eta$ that controls the size of the attaching regions of the handles in $W_{0}$. The link of isotropic attaching spheres $\partial L_{k-1}$ has a neighborhood of the form $\partial L_{k-1}\times D^{2n-k}$. We let $\eta>0$ control the size of these attaching regions and write $\partial W_{\eta}$ for the boundary of the Weinstein domain with $\eta$-sized handles. Using the Liouville flow in local models it is clear that if $\eta'<\eta$ then there is a topologically trivial symplectic cobordism with positive end $\partial W_{\eta}$ and negative end $\partial W_{\eta'}$ which can be taken to be standard in the middle of each handle.

\begin{lemma}
	For generic attaching spheres the following holds.	
	For any $\mathfrak{a}>0$ there exists $\eta_{0}>0$ such that for all $\eta<\eta_{0}$, any Reeb orbit of $\partial W_{\eta}$ of action $<\mathfrak{a}$ lies in the middle of some handle. Furthermore, for all $\eta<\eta'<\eta_{0}$ rigid holomorphic cylinders in the cobordism gives a natural one to one correspondence between Reeb orbits of action $< \mathfrak{a}$ in $\partial W_{\eta'}$ and $\partial W_{\eta}$. Finally, there exists a contact form on $\partial W_{\eta}$, arbitrarily close to the standard contact form in the middle of each handle such that the minimal Conley--Zehnder index of an orbit in a $k$-handle is $(n-k)+1$.	
\end{lemma}   

\begin{proof}
	A subcritical isotropic sphere generically has no Reeb chords. Therefore for any $\mathfrak{a}>0$ there exists $\eta$ such that any Reeb flow line starting in an $\eta$-neighborhood of the sphere returns only if it has action $>2\mathfrak{a}$. It follows from this that all Reeb orbits must lie inside the handles. The claim then follows from well-know properties of the standard contact sphere, which sits in the middle of each handle.  
\end{proof}

We call contact forms on subcritical manifolds of the form above \emph{index definite}. Note that if $\Lambda\subset \partial W$ is a Legendrian sphere in the \emph{index definite} boundary of a subcritical Weinstein manifold of dimension $2n>2$, then its dg-algebra can be computed without any anchoring: punctured spheres in $W$ have minimal dimension $n-(n-1)+1+(n-3)=n-1>0$, see e.g., \cite[Appendix A]{CEL} for the well-known dimension formula. 

\subsection{Reeb dynamics and attaching spheres of $V$-handles}
Let $V$ be a Weinstein $(2n-2)$-manifold with handle decomposition $h$. Consider the contactization of $V$, i.e., the contact manifold $\R\times V$ with contact form $d\zeta+\lambda$, where $\zeta$ is a linear coordinate on $\R$. We will construct a cobordism $V\times T^{\ast} [-1,1]$ with negative end $(\R\times V)\times \{-1,1\}$ and positive end the contact manifold obtained by removing $((-\epsilon,\epsilon)\times V)\times \{-1,1\}$ from $(\R\times V)\times \{-1,1\}$ and joining the boundaries by $(V\times S^{\ast}_{\epsilon} [-1,1])\cup \left(\partial V\times D^{\ast}_{\epsilon}I|_{\{-1,1\}}\right)$, where $S^{\ast}_{\epsilon} [-1,1]$ and $D^{\ast}_{\epsilon} [-1,1]$ denotes the $\epsilon$-sphere and $\epsilon$-disk cotangent bundles, respectively. The cobordism will furthermore come with a natural handle decomposition induced by $h$. In order to get a more precise description we will use an explicit model of the handle that we describe next.

Consider symplectic $\R^{2}$ with coordinates $(x,y)$ and symplectic form $dx\wedge dy$. Consider the product $V\times\R^{2}$ with symplectic form $\omega=d\lambda+dx\wedge dy$. Then 
\[ 
Z=z + 2x\partial_{x}-y\partial_y
\]
is a Liouville vector field for $\omega$ and we consider $H_{\delta}(V)\approx V\times\R^{2}$ as an exact symplectic cobordism with positive and negative ends the contact hypersurfaces
\begin{equation}\label{eq:cobGdelta} 
G_{\pm\delta}(V)=\{(x,y,v)\in\R^{2}\times V\colon x^{2}-\tfrac12 y^{2}+\phi(v)=\pm\delta\}.
\end{equation}
The induced contact form on $G_{\delta}(V)$ is
\[ 
\alpha_{\pm\delta}=(2xdy+ydx+\lambda)|_{G_{\pm\delta}}
\]
and the corresponding Reeb vector field
\begin{equation}\label{eq:Reebhandle} 
R_{\pm\delta}=\beta(x,y,v)(2x\partial_{y}+y\partial_{x})+r_{\lambda}(x,y,v),
\end{equation}
where $\beta(x,y,v)$ is a smooth function, non-zero if $(x,y)\ne (0,0)$ and $r_{\lambda}$ is a smooth vector field in the kernel of $d\lambda|_{\{\phi(v)=-x^{2}+ \tfrac 12y^{2}\pm\delta\}}$ at points where $z\ne 0$ and equal to zero where $z=0$. 
\begin{lemma}\label{l:Reebinhandle}
The Reeb vector field $R_{\delta}$ has the following properties. The subset of $G_{\delta}$ given by $(x,y)=(0,0)$ and $\phi(v)=\delta$ is invariant under the flow of $R_{\delta}$ and along this subset $R_{\delta}$ agrees with the Reeb vector field of $\partial V$. 
\end{lemma}

\begin{proof}
	Clear from \eqref{eq:Reebhandle}.
\end{proof}

Now take $V=V_{0}$, the subcritical part of $V$. Write $l=\bigcup_{j} l_{j}$ for the core-disks of the top handles $h_{j}$ in $h$ and $\partial l\subset \partial V_{0}$ for their boundaries which are the attaching $(n-2)$-spheres for $h_{j}$. Consider 
\[ 
\overline{\partial l}=\left\{(x,y,v)\in\R^{2}\times V_{0}\colon x=0, v\in \partial l\times[0,\infty), (0,y,v)\in G_{\delta}(V_{0}) \right\}.
\]
It has the following properties (below the (contact homology) \emph{grading} of a Reeb orbit is the dimension of the moduli space of holomorphic planes with positive asymptotic at the orbit).
\begin{lemma}\label{l:middlesubcritLeg}
	The subset $\overline{\partial l}\subset G_{\delta}$ is a smooth Legendrian submanifold of topology $\partial l\times \R$. The Reeb chords of $\overline{\partial l}$ all lie over $(x,y)=0$ and are in natural 1-1 grading preserving correspondence with the Reeb chords of $\partial l\subset\partial V_{0}$. Furthermore, the Reeb orbits in $G_{\delta}$ are in natural 1-1 correspondence with Reeb orbits in $\partial V_{0}$, where the grading of an orbit in $G_{\delta}$ is one above the grading of the corresponding orbit in $\partial V_{0}$. 
\end{lemma}

\begin{proof}
	With Reeb vector field $R_{\delta}$ as in \eqref{eq:Reebhandle}, the $(x,y)$-coordinate of any Reeb flow line not starting at $(x,y)=0$, moves along the curves $\{2x^{2}-y^{2}=\mathrm{const}\}$ with non-zero speed, and the $(x,y)$-coordinate of flow lines starting at $(x,y)=(0,0)$ is fixed. Since the restriction of $R_{\delta}$ to $(x,y)=0$ agrees with the Reeb vector field in $\partial V$ it follows that Reeb chords are in 1-1 correspondence as claimed. It remains to show that this correspondence is index preserving.
	
	Consider a Reeb chord $c$ of $\partial l\subset \partial V_{0}$. The degree $|c|$ of $c$ is $-|c|=\mu(c)-1$, where $\mu(c)$ is the Maslov index of the tangent planes to the Legendrian along a path from the top end point of $c$ to the bottom endpoint, followed by the linearized Reeb flow along the chord and finally closed up by a positive rotation along the K\"ahler angle. Including the chord as a chord of $\overline{\partial l}$, the only change comes from the linearized Reeb flow in the additional $\R$-direction. The Reeb vector field is $y\partial_{x}$ which means that $y$-axis at the bottom endpoint arrives rotated slightly in the negative and closed up by a small positive rotation this contributes zero to the Maslov index. It follows that the gradings of $c$ as a chord of $\partial l$ and $\overline{\partial l}$ agree.
	
	The statement on Reeb orbits is similar. The difference in grading arises as follows: as with the Maslov index of the chord, the Conley--Zehnder index $\mathrm{CZ}$ of an orbit does not change from $\partial V$ to $G_{\delta}$. Since the grading of an orbit is $\mathrm{CZ}+(n-3)$, where the dimension of the contact manifold is $2n-1$, the grading then increases by one.   
\end{proof}

Consider now adding $V_{0}\times D^{\ast}_{\epsilon}[-1,1]$ as above to $V\times\R\times\partial [-1,1]$ along $V_{0}\times(-\epsilon,\epsilon)$. Smoothing the union of the two copies of $l$ in $V\times\R\times\partial I$ and $\overline{\partial l}$, we get a collection of Legendrian spheres $\Sigma(h)$ in the upper contact boundary.

\begin{lemma}\label{l:conehandles}
	The Legendrian spheres $\Sigma(h)$ are attaching spheres for $V\times\R^{2}$ with Liouville field as in the cobordism with positive end $G_{\delta}$ and negative end $G_{-\delta}$, see \eqref{eq:cobGdelta}. 
\end{lemma}	

\begin{proof}
	Consider an exhausting function $\phi$ for $V$ such that $V_{0}=\phi^{-1}([0,\tfrac12\delta])$ with critical level for critical points of maximal index at $\tfrac32\delta$. Then the level set $\delta$ gives the contact manifold with attaching spheres $\Sigma(h)$. The critical points in the core disks lie at the level $\tfrac32\delta$ and sub-level sets of levels $>\frac32\delta$ give $V\times\R^{2}$.
\end{proof}

\subsection{Holomorphic curves in $V$-handles}\label{ssec:holcurveinhandle}
We next consider almost complex structures on $H_{\delta}(V)$. Let $J_{V}$ be a handle adapted almost complex structure on $V$, see Section \ref{sec:Weinhand}. Write $z$ for the Liouville vector field of $V$ and let $r=J_{V}z$. We think of $H'_{\delta}(V)=H_{\delta}(V)\setminus (\{x=0\}\times \mathrm{skeleton}(V))$ as the symplectization of $\partial H_{\delta}$. We will define an almost complex structure on $G_{\delta}$ that takes the Liouville vector field to the Reeb vector field and then extend it to all of $H_{\delta}'(V)$ by translation along the Liouville vector field. 

Consider the tangent space to $G_{\delta}(V)$. 
\begin{itemize}
\item At $(x,y)=0$, the tangent space is spanned by $\ker d\phi$, $\partial_{x}$, and $\partial_{y}$. We define $J$ as $J_{V}$ on $TV$ and $J\partial_{x}=\partial_{y}$.  
\item At points $(x,y,v)\in G_{\delta}(V)$, where $v$ is not a critical point of $\phi$ and $(x,y)\ne 0$, the tangent space is the sum of the $(2n-3)$-dimensional space $\ker(d\phi)$ and the one-dimensional subspaces generated by vector field $(\sqrt{x^{2}+y^{2}})^{-1}(y\partial_{x}+2x\partial_{y})$ and the vector field $-(4x^{2}+y^{2})z+d\phi(z)(2x\partial_{x}-y\partial_{y})$. Here we define $J$ as the restriction of $J_{V}$ on the contact hyperplane in $\ker\phi$ and 
\[ 
J z=\frac{1}{\omega(z,r)+4x^{2}+y^{2}} r,\quad J(2x\partial_{x} - y \partial_y)=\frac{1}{\omega(z,r)+4x^{2}+y^{2}} (2x \partial_y + y \partial_x).
\] 
\item At points $(x,y,v)$ where $v$ is critical for $\phi$ we define $J$ to equal $J_{V}$ on $TV$ and $J(2x\partial_{x}-y\partial_{y})=\frac{1}{4x^{2}+y^{2}}(2x\partial_{y}+y\partial_{x})$. (Note that, $v$ critical and $(x,y,v)\in G_{\delta}$ imply $(x,y)\ne 0$.)
\end{itemize}
Since $J_{V}$ is handle adapted it follows that $J$ is smooth. We finally extend $J$ by translation along the Liouville vector field. Let $\pi\colon H_{\delta}\to\R^{2}$ be the projection to the $(x,y)$-plane.

\begin{lemma}\label{l:complexproj}
The hypersurfaces $V_{(x,y)}=\pi^{-1}(x,y)$ are $J$-complex. In particular, any holomorphic curve in the symplectization of $G_{\delta}$ not contained in $V_{(x,y)}$ intersects $V_{(x,y)}$ positively and the intersection number is locally constant in $(x,y)$.
\end{lemma}

\begin{proof}
The splitting $T H_{\delta}(V)=TV\oplus T\R^{2}$ is preserved along the flow lines of the Liouville field $Z=z+2x\partial_{x}-y\partial_{y}$ and $J$ takes $TV_{(x,y)}$ to $TV_{(x,y)}$ by definition. 
\end{proof}

\begin{remark}
The above construction gives an almost complex structure in the symplectization of the positive end of the symplectic cobordism $H_{\delta}(V)$, represented by $H'_{\delta}(V)$. When we study curves in the cobordism rather than in the symplectization we extend the almost complex structure from a cut off version of the symplectization to the compact part of the cobordism by interpolating between the symplectization complex structure and an almost complex structure in the compact part over a finite interval in the Liouville direction near the cut-off.
\end{remark}

We next consider holomorphic curves in the $V_{0}$-handle. Recall that we have an index definite Weinstein structure on $V_{0}$. For $\dim(V_{0})=2n-2>2$ this means that the minimal grading of any orbit is $1$ and therefore the dg-algebra of $\partial l\subset\partial V_{0}$ is defined without anchoring. We start with this case, where there will be a canonical identification of dg-algebras. In the lowest dimensional case, $\dim(V_{0})=2$ the basic orbit has grading $0$ and its effects must be taken into account, see Remarks \ref{r:dimV_0=2} and \ref{r:dimV_0=2discussion}. 

\begin{lemma}\label{l:middlesubcrit}
	Assume $n>2$. For $\mathfrak{a}>0$, any $J$-holomorphic curve in the symplectization $\R\times G_{\delta}$ with boundary on $\R\times \overline{\partial l}$ and with positive puncture at a Reeb chord over $y=0$ corresponding to a Reeb chord of $\partial l$ of action $<\mathfrak{a}$, lies entirely over $(x,y)=0$. Also, there is a natural 1-1 correspondence between such rigid holomorphic disks and rigid $J_{V_{0}}$-holomorphic disks in $\R\times \partial V_{0}$ with boundary on $\R\times\partial l$.
\end{lemma}

\begin{proof}
	If a curve intersects $V_{(x,y)}$ but is not contained in it then by Lemma \ref{l:complexproj} it is unbounded outside $(x,y)=0$ which shows it would have some positive asymptote not at $(x,y)=0$ in contradiction to our assumption. Consequently, all rigid curves actually are contained in $(x,y)=0$. The correspondence between curves is obvious, it remains only to show that the rigid disks are transversely cut out. This is straightforward, the linearization corresponds to a stabilization with small angles at the punctures, i.e., multiplication by $(\C,\R)$ followed by small deformation so that the $\R$-component at the incoming boundary component is rotated a small positive (negative) angle at positive (negative) punctures, which is easily seen to be an isomorphism.
\end{proof}

\begin{remark}\label{r:dimV_0=2}
We consider the counterpart of Lemma \ref{l:middlesubcrit} when $n=2$. In this case $V_{0}$ is a $2$-disk and $\partial l$ is a collection of points in its $S^{1}$-boundary. This type of handle was studied in \cite{ENg}. It was shown that the resulting holomorphic curves stay inside the handle and the dg-algebra was computed and was called the internal algebra \cite[Section 2.3]{ENg}. This algebra is canonically isomorphic to $CE^{\ast}(\partial l,V_{0})$, see \cite[Sections 4 and 5]{ENg}.
\end{remark}

\begin{remark}\label{r:dimV_0=2discussion}
We give a brief discussion of geometric aspects of the case $n=2$. Disks contributing to the differential in $CE^{\ast}(\partial l;D^{2})$ are anchored at the basic Reeb orbit of index $0$. Including $D^{2}$ in the middle of the $4$-dimensional handle, the grading of the Reeb orbit increases to $1$, there is no anchoring disks anymore, and the upper levels (in the symplectization) in the anchored curves are now of dimension $\le 0$. In order to achieve transversality one must break the symmetry of the Lagrangian and the Reeb chords are no longer contained in $(x,y)=0$. In particular, the trivial holomorphic strips over the Reeb chords that extend in the Liouville direction now appear in the handle model $H_{\delta}$ as thin strips around $\{y=0\}$ stretching to infinity. The perturbation breaks the symmetry and the strips go to infinity in only one direction. The disk family with a positive asymptotic at the Reeb chord going once around lie arbitrarily close to the broken disk of the augmented curve in the middle and the basic disk over the half axis chosen by the perturbation. Counts of other anchored disks can be concluded formally from properties of the differential. Geometrically, they all lie close to the anchored curve in the middle with half lines of standard disks attached.    
\end{remark}

\begin{remark}
Curves in fibers $V_{(x,y)}$, $(x,y)\ne (0,0)$ as in Lemma \ref{l:middlesubcrit} are not fixed by the Liouville flow. To find their asymptotics at the negative end one can consider the intersection with the level surface $G_{\delta'}$ as we translate the curve to infinity in the positive Liouville direction.
\end{remark}

Consider now adding $V_{0}\times D^{\ast}_{\epsilon}[-1,1]$ as above to $V\times\R\times\partial [-1,1]$ along $V_{0}\times(-\epsilon,\epsilon)$. Smoothing the union of the two copies of $l$ in $V\times\R\times\partial I$ and $\overline{\partial l}$, we get a collection of Legendrian spheres $\Sigma(h)$ in the upper contact boundary.
	
\begin{corollary}\label{c:dgsubalg}
 The Chekanov--Eliashberg dg-algebra of $\Sigma(h)$ is canonically isomorphic to the dg-algebra of $\partial l\subset\partial V_{0}$, using the identification of Reeb chords and holomorphic disks in Lemma \ref{l:middlesubcrit} for $n>2$ and the corresponding identification from Remark \ref{r:dimV_0=2} for $n=2$. \qed
\end{corollary}

\section{Chekanov--Eliashberg dg-algebras and Legendrian surgery formulas}
In this section we define the Chekanov--Eliashberg dg-algebra of a Weinstein $(2n-2)$-domain $(V,h)\subset \partial W$ with handle decomposition $h$. It follows from the definition and the Legendrian surgery formula that the dg-algebra is isomorphic to the endomorphism algebra of co-core disks dual to the top-dimensional core-disks of $V$ in the partially wrapped Fukaya category with a stop at $V$. We also show that the dg-algebra can be decomposed into an exterior piece generated by Reeb chords in $\partial W$ connecting the $(n-1)$-dimensional core-disks $l$ of top-dimensional handles in $h$ and the dg-algebra of their attaching $(n-2)$-spheres $\partial l$ in $\partial V_{0}$, the subcritical part of $V$, which here appears as a sub-algebra, compare Corollary \ref{c:dgsubalg}. 

\subsection{Definition of Chekanov--Eliashberg dg-algebras}\label{sub:formal_def}
Let $W$ be a Weinstein $2n$-manifold and let $(V,h)\subset \partial W$ be a Legendrian embedding of the Weinstein $(2n-2)$-domain $V$ with handle decomposition $h$. Then a small neighborhood of $V$ in $\partial W$ can be identified with $V\times(-\epsilon,\epsilon)$, where the second factor is along the Reeb flow. We define $W_{V}$, $W$ stopped at $V$, to be the Weinstein cobordism with negative end $V\times\R$ obtained by attaching a $V$-handle, i.e., an $\epsilon$-neighborhood $V\times D^{\ast}_{\epsilon}[-1,1]$ of $V\times[-1,1]$ in $V\times T^{\ast}[-1,1]$, along $V\times(-\epsilon,\epsilon)$ to $\partial W\sqcup (\R\times V)$. (As mentioned in Section \ref{sec:intr}, we think of $V$ in the negative end $\R\times V$ as the Weinstein manifold which is the completion of the embedded Weinstein domain.)

\begin{lemma}\label{l:handlesandnoorbits}
The negative end of the cobordism $W_{V}$ have no closed Reeb orbits. If $W$ comes with a handle decomposition $H$ then $h$ induces a handle decomposition of $W_{V}$ with handles $H\cup h[1]$, where $h[1]$ is the set of handles of $h$ with dimension shifted up by $1$ and geometrically described in Lemma \ref{l:conehandles}. 
\end{lemma} 

\begin{proof}
The first statement holds by construction since the Reeb flow on $\R\times V$ has no orbits. The second statement follows from Lemma \ref{l:conehandles}. 
\end{proof}

Let $(V_{0},h_{0})$ be the subcritical part of $(V,h)$. Consider the subset $W_{V}^{0}\subset W_{V}$ which is obtained by attaching a $V_{0}$-handle to $W\sqcup (\R\times V)$. Lemma \ref{l:conehandles} then gives a link of Legendrian attaching $(n-1)$-spheres $\Sigma(h)\subset\partial_{+}W_{V}^{0}$, where $\partial_{+}W_{V}^{0}$ denotes the positive boundary of $W_{V}^{0}$. Here there is one component $\Sigma(h_{j})$ for each top-dimensional handle $h_{j}$ of $V$ and attaching Weinstein $n$-handles to $W_{V}^{0}$ along $\Sigma(h)$ we obtain $W_{V}$. As in Lemma \ref{l:handlesandnoorbits}, there are no Reeb orbits in the negative end of the cobordism $W_{V}^{0}$. We can therefore use it for anchoring holomorphic disks and the ordinary Chekanov--Eliashberg dg-algebra $CE^\ast(\Sigma(h);W_{V}^{0})$ is defined as in \cite{BEE,EkholmHol,EL}. We define the Chekanov--Eliashberg dg-algebra of $(V,h)\subset W$ to be the dg-algebra of the attaching link $\Sigma(h)$. More precisely, we have the following.

\begin{definition}\label{dfn:def_CE_of_V}
Let $(V,h)\subset\partial W$ be a Legendrian embedding. The \emph{Chekanov--Eliashberg dg-algebra} of $(V,h)$, $CE^\ast((V,h);W)$ is
\[ 
CE^\ast((V,h);W) \ := \ CE^\ast(\Sigma(h);W_{V}^{0}),
\]
where the link of Legendrian spheres $\Sigma(h)\subset \partial_{+}W_{V}^{0}$ is as described above and where the right hand side is the standard Legendrian invariant as defined in \cite{BEE,EkholmHol,EL}.
\end{definition}

\begin{proof}[Proof of Theorem \ref{t:basic}]
	Theorem \ref{t:basic} follows from the Legendrian surgery formula \cite{BEE,EkholmHol,EL} and the definition of $CE^{\ast}((V,h);W)$ as the Legendrian dg-algebra of the link of attaching spheres $\Sigma(h)$ of core disks dual to $C(h)$.
\end{proof}

\begin{remark}\label{r:overviewLegsurg}
We give a short overview of the proof of the Legendrian surgery formula used in proof of Theorem \ref{t:basic}. The result is stated in \cite[Theorem 5.8 and Remark 5.9]{BEE} and proved with details as \cite[Theorem 83]{EL} using technical results of \cite{EkholmHol}. The argument is as follows. The first step is to relate the Reeb dynamics in the contact boundary of the Weinstein manifold before handle attachment to that after. Outside a small neighborhood of the Legendrian attaching sphere the Reeb flows are identical and inside the handle the Reeb flow follows a standard model which behaves qualitatively like the geodesic flow in the unit cotangent bundle of the flat disk. Here the 0-section should be thought of as the core disk and the cotangent fiber at the center of the disk as the co-core disk of the handle. Shrinking the handle, one finds a 1-1 correspondence between Reeb chords of the co-core disks after the surgery and words of Reeb chords of the attaching sphere before the surgery, see \cite[Theorem 1.2]{EkholmHol}. The second step is to construct an $A_{\infty}$-homomorphism between the wrapped Floer cohomology of the co-core disks and the dg-algebra of the Legendrian attaching spheres, where we view the latter as a chain complex generated by composable words of chords with differential induced by the dg-algebra differential, product given by concatenation of words, and all higher operations trivial. This chain map counts holomorphic disks in the cobordism. The disks have two boundary punctures mapping to the intersection points between core and co-core disks, positive punctures at Reeb chords of the co-cores and negative punctures at Reeb chords of the attching spheres. Using a version of wrapped Floer cohomology without Hamiltonian perturbation, see \cite[Appendix B.1]{EL}, one shows that this map is a chain map, see \cite[Theorem 83, first part]{EL}. The final step of the proof is to show that the cobordism map induces is a quasi-isomorphism. This is proved usuing a natural action filtration with respect to which the chain map is triangular with $\pm 1$ on the diagonal. The fact that the chain map respects the filtration is essentially an application of Stokes theorem. The holomorphic disks that gives units on the action diagonal is constructed using disks in local models, gluing and action to control how disks in resulting moduli spaces break, see \cite[Theorem 1.3 (1)]{EkholmHol} for the construction of the disks and \cite[Theorem 83, second part]{EL} for how to deduce the quasi-isomorphism result.     
\end{remark}

\subsection{Generators}
Definition \ref{dfn:def_CE_of_V} gives $CE^{\ast}((V,h);W)$ as the usual Legendrian dg-algebra of a link  $\Sigma(h)$ and as such, it is generated by Reeb chords. The Reeb chord generators can be described as follows. Recall that we write $l$ for the union of the core $(n-1)$-disks of the top handles in $V$ and $\partial l\subset\partial V_{0}$ for their Legendrian attaching spheres in the contact boundary $\partial V_{0}$ of the subcritical part of $V$.
\begin{lemma}\label{l:generators}
For any given action level $\mathfrak{a}>0$, for all sufficiently small handles (i.e., sufficiently small $\epsilon>0$ in the handle $V_{0}\times D^{\ast}_{\epsilon} I$), after arbitrarily small perturbation of $(V,h)$, Reeb chords of $\Sigma(h)$ of action $<\mathfrak{a}$ in $\partial W_V^{0}$ are in natural 1-1 correspondence with Reeb chords connecting core-disks in the top handles $l\subset \partial W$ and Reeb chords of $\partial l\subset \partial V_{0}$. Furthermore, all chords can be assumed transverse.
\end{lemma}

\begin{proof}
The subcritical part of the skeleton of $V$ has dimension $n-2$. Hence for any $\mathfrak{a}>0$ there is a small neighborhood $V_{0}$ of the subcritical part of the skeleton such that no Reeb flow line starting on $l\times \{\frac12\epsilon\}\subset V\times(-\epsilon,\epsilon)$ of action $<\mathfrak{a}$ hits $V_{0}\times(-\epsilon,\epsilon)$. The result then follows from the definition of $\Sigma(h)$ and Lemma \ref{l:conehandles}.
\end{proof}

\subsection{A basic property of the differential}
We next consider the differential of the dg-algebra. The differential is defined in the standard way in terms of counts of holomorphic disks in $W_{V}^{0}$ with boundary condition $\Sigma(h)$.   

Assume that $(V,h)$ is in general position so that Lemma \ref{l:generators} holds. Lemma \ref{l:middlesubcrit} shows that $CE^{\ast}((V,h);W)$ has a subalgebra canonically identified with the dg-algebra of the top-dimensional attaching spheres of $V$. To compute the differential it remains to describe its action on Reeb chords outside the $V_{0}\times D^{\ast}_{\epsilon} [-1,1]$-handle. We will describe a concrete way of doing this in the case when $\partial W$ is a contactization in Section \ref{ssec:contactizations}. Here we establish a useful general property of curves contributing to the differential, such curves cannot `cross' the handle. Let $\mathfrak{a}>0$ be given and fix $\epsilon>0$ so that Lemma \ref{l:generators} holds.

\begin{lemma}\label{l:nocross}
No holomorphic curve in $\R\times\partial_{+} W_{V}^{0}$ with positive puncture at a Reeb chord of $\Sigma(h)$ of action $<\mathfrak{a}$ can cross the handle. In other words, such a curve lies entirely in $W\cup (V_{0}\times D_{\epsilon}^{\ast}[-1,0])$. 
\end{lemma} 

\begin{proof}
To see this note that inside the handle such a curve must intersect the hypersurface $\{y=0\}$ which is foliated by the $J$-complex submanifolds $V_{(x,0)}$. It follows by positivity of intersections that the curve is either contained in a $J$-complex submanifold $V_{(x,0)}$ or it intersects every fiber over $(x,0)$ non-trivially. This means the curve must have a positive asymptotic inside the handle (recall the component of the Liouville vector fields along the $x$-axis points outwards) not over the origin, which is incompatible with being a curve with positive puncture at a Reeb chord of $\Sigma(h)$ of action $< \mathfrak{a}$ by Lemmas \ref{l:handlesandnoorbits} and \ref{l:generators}. It follows that the curve cannot cross the handle.  
\end{proof}

\begin{remark}
The proof of Lemma \ref{l:nocross} shows that any holomorphic curve that passes through the handle must have positive asymptotics at a Reeb chord or orbit that goes through the handle. In case $V$ have critical handles there are such chords and orbits. In the subcritical index definite case they are ruled out below any given action level for sufficiently thin handles by Lemmas \ref{l:handlesandnoorbits} and \ref{l:generators}. 
\end{remark}

\section{Non-singular Legendrians treated as singular Legendrians}\label{sec:newandold}
In this section we show that the new definition of Chekanov--Eliashberg dg-algebras agrees with the standard definition for smooth Legendrians. As a consequence we conclude that \cite[Conjecture 3]{EL} holds.  

Let $\Lambda\subset\partial W$ be a smooth connected Legendrian submanifold with a handle decomposition $h$ with a single top handle and let $V=D^{\ast}_{\epsilon}\Lambda$ denote a small neighborhood of $\Lambda$ in $T^{\ast}\Lambda$. By the Darboux theorem for Legendrians there is a Legendrian embedding $V\to W$ canonically associated to $\Lambda$ for all sufficiently small $\epsilon>0$. 

Recall the construction of $CE^{\ast}((V,h);W)$: $\Sigma(h)\subset\partial W_{V}^{0}$ denotes the attaching sphere of the top handle of $W_{V}$. It consists of two copies of the core disk $l$ of the top handle joined by $[-1,1]\times\partial l$ across the handle. We next note that the cobordism $W_{V}^{0}$ contains a natural Lagrangian cobordism $Q'$ with topology $(\Lambda\times [-1,1])\setminus D^{n}$, where $D^{n}$ is a disk with boundary $\Sigma(h)$. The positive boundary of $Q'$ is $\Sigma(h)$ and the negative boundary is $\Lambda\times\{-1,1\}$. 

We will consider a subset $Q\subset Q'$ with the topology of $\Lambda\times [-1,\delta]$. More precisely, we take $Q=Q'\cap V\times D^{\ast}_{\epsilon}[-1,\delta]$, where $\delta>0$. (Here we think of the critical points of the cobordism as sitting in $V\times 0$.)

\begin{figure}[!htb]
	\includegraphics{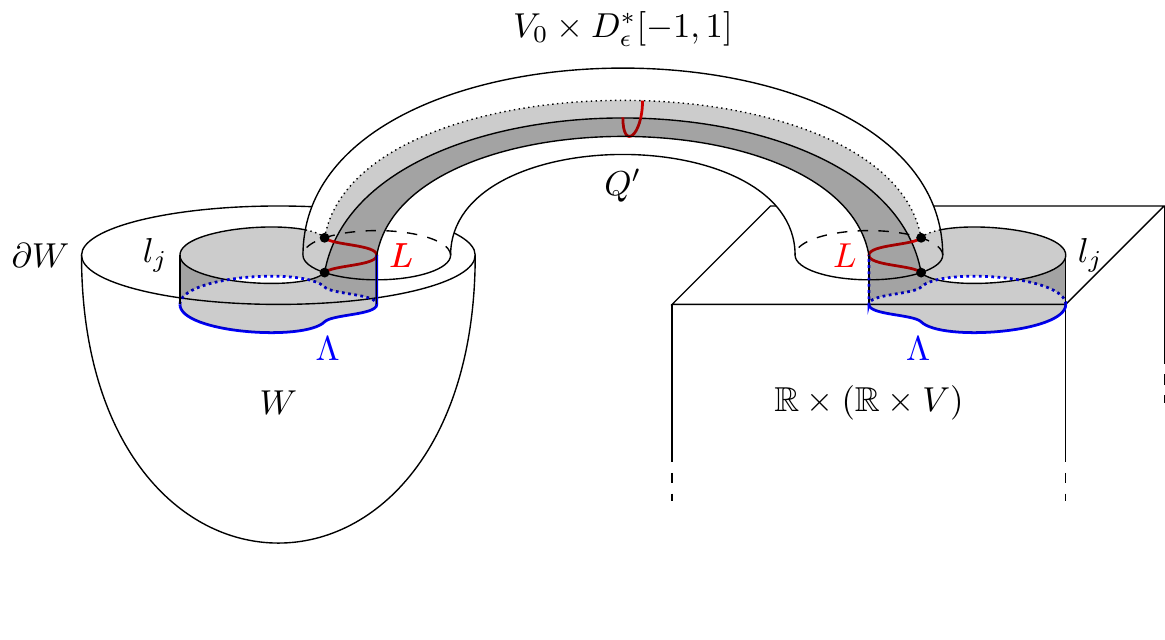}
	\caption{The Lagrangian cobordism $Q'$ with positive boundary $\Sigma(h)$ and negative boundary $\Lambda \times \{-1,1\}$ is shaded in gray. The handles of $\Lambda$ below maximal dimension gives a Lagrangian filling $L$ of $\partial l$ in $V_0$.}
	\label{fig:Q_cobordism}
\end{figure}

To relate the two definitions of dg-algebras we first note that 
\[
CE^{\ast}(\Lambda,C_{-\ast}(\Omega\Lambda))=CE^{\ast}(\Lambda,C_{-\ast}(\Omega Q)),
\] 
where the map on coefficients is induced by inclusion. We then note that there is a natural cobordism map
\[ 
\Phi\colon CE^{\ast}(\Sigma(h);W_{V}^{0})\longrightarrow CE^{\ast}(\Lambda,C_{-\ast}(\Omega Q)),
\] 
defined by moduli spaces of holomorphic disks with positive puncture in the positive end and negative punctures in the negative and `loop space coefficients' in the cobordism $Q$. 

More formally, let $a$ be a Reeb chord of $\Sigma(h)$ and let $\mathcal{M}(a)$ denote the moduli space of holomorphic disks in the cobordism $W_{V}^{0}$ with boundary on $Q'$ and negative punctures at Reeb chords of $\Lambda$. Any component $\mathcal{M}(a;b_{1}\dots b_{m})$ of such a moduli space where the negative punctures are at Reeb chords $b_{1},\dots,b_{m}$ defines a sum of monomials in $CE^{\ast}(\Lambda,C_{-\ast}(\Omega Q'))$ as follows. The fundamental chain of the moduli space gives a chain of based loops in the product, with one factor for each boundary segment. Applying the Alexander--Whitney diagonal approximation as in \cite{EL} we get a sum of monomials 
\[ 
\sigma_{0}b_{1}b_{2}\dots b_{m} +b_{1}\sigma_{1}b_{2}\dots b_{m}+\dots+b_{1}\dots b_{m}\sigma_{m}, 
\]
in the dg-algebra, where $\sigma_{j}$ are chains of based loops in $Q'$, see \cite{EL}. We write $[\mathcal{M}(a)]$ for the sum of all such elements over all components in the moduli space with positive puncture at $a$ and define
\[ 
\Phi(a) \ = \ [\mathcal{M}(a)] \ \in \ CE^{\ast}(\Lambda,C_{-\ast}(\Omega Q')).
\]
     
\begin{lemma}
The map $\Phi$ is a chain map, $[\mathcal{M}(a)]\in C_{-\ast}(\Omega Q)\subset C_{\ast}(\Omega Q)$ and therefore $\Phi$ defines a chain map $CE^{\ast}(\Sigma(h);W_{V}^{0})\to CE^{\ast}(\Lambda,C_{-\ast}(\Omega Q))$.
\end{lemma}

\begin{proof}
The proof of the chain map property follows as in \cite[Proposition 21]{EL} once we know that all holomorphic disks with positive puncture at a Reeb chord of $\Sigma(h)$ has its boundary in $Q\subset Q'$. This follows from Lemma \ref{l:nocross}.
\end{proof}

We will next show that $\Phi$ is a quasi-isomorphism. To this end we use a particular representation of the loop space of $\Lambda$. More precisely, we replace $\Lambda$ by a space $\bar\Lambda$ in which a disk containing the core of the top cell $l$ is identified to a point. We first consider the restriction of the map $\Phi$ to the short Reeb chords inside the handle. It follows from Lemma \ref{l:middlesubcrit} that this map is the natural curve counting map from $CE^{\ast}(\partial l;\partial V_{0})$ that associates to a Reeb chord the chain of loops in $\bar\Lambda$ carried by the moduli space of disks with a positive puncture at that Reeb chord. We call this map $\psi$.

\begin{lemma}\label{l:surgeryloopspace}
The map
\[ 
\psi\colon CE^{\ast}(\partial l;\partial V_{0})\longrightarrow C_{-\ast}(\Omega\bar\Lambda)
\]
is a quasi-isomorphism.
\end{lemma}

\begin{proof}
Consider the wrapped Floer cohomology of the fiber $F$ in the middle of the top handle. Note that $V=T^{\ast}\Lambda$. Recall the standard quasi-isomorphism
\[ 
\alpha\colon CW^{\ast}(F;V)\longrightarrow C_{-\ast}(\Omega\bar\Lambda),
\]
see e.g., \cite{AbouzaidBased,AsplundFiber}, and the Legendrian surgery isomorphism, \cite{BEE,EkholmHol,EL},
\[ 
\beta\colon CW^{\ast}(F;V)\longrightarrow CE^{\ast}(\partial l;\partial V_{0}).
\]
By SFT-stretching,
\[ 
\alpha=\psi\circ\beta.
\]
Since $\alpha$ and $\beta$ are both quasi-isomorphisms so is $\psi$.
\end{proof}

\begin{proof}[Proof of Theorem \ref{t:loopspacecoeff}]
After Lemma \ref{l:surgeryloopspace}, the standard action argument applies: Lemma \ref{l:surgeryloopspace} gives the desired small action isomorphism at the action level of the chords in the middle of the handle, as we increase the action, trivial strips over the Reeb chords of $l$ shows that the chain map $\Phi$ is triangular with respect to the action filtration with $\pm 1$ on the diagonal.  
\end{proof}

Theorem \ref{t:loopspacecoeff} can be specialized in various ways. Here we consider the case when there is an augmentation of the dg-algebra corresponding to the loop space. If $\epsilon\colon CE^{\ast}(\partial l;\partial V_{0})\to\C$ is an augmentation then we can form the $\epsilon$-partially linearized Chekanov--Eliashberg dg-algebra $CE^{\ast}((V,h);W;\epsilon)$ generated by chords of $l$ only and with differential given by the differential in $CE^{\ast}((V,h);W)$ followed by $\epsilon$ on short chords in the subalgebra $CE^{\ast}(\partial l;V_{0})$, compare \cite[Section 4.6]{BEE}.

\begin{corollary}\label{cor:smoothusual}
	 If $\epsilon=\epsilon_{L}$, where $L$ is a Lagrangian filling of $\partial l$ in $V_{0}$ then $CE^{\ast}((V,h);W;\epsilon)$ is isomorphic to the usual Legendrian dg-algebra of $\Lambda$, where $\Lambda$ is obtained by capping $l$ off by $L$.
\end{corollary}

\begin{proof}
	 The chain map relating the algebras is the composition of $\Phi$ with the map that takes degree zero chains to $1$ and other chains to $0$. The quasi-isomorphism statement follows as in the proof of Theorem \ref{t:loopspacecoeff}, by restricting attention to long chords only.
\end{proof}

\section{Cut and paste}\label{sec:cut_and_paste}
In this section we prove Theorem \ref{t:vankampen} and consider its consequences for stop removal. 

\subsection{A natural push-out diagram}
Assume $(V,h)$ is Legendrian embedded in the boundary of two Weinstein manifolds $W$ and $W'$, respectively. Consider $W_{V}$ and $W'_{V}$ as defined in Section~\ref{sub:formal_def}. Consider the Weinstein manifold $W \#_{V} W'$ obtained by connecting $W$ to $W'$ by a $V$-handle. More precisely, we construct $W\#_{V} W'$ by attaching $V\times D^{\ast}_{\epsilon}[-1,1]$ exactly as in the construction of $W_{V}$ from $W$ and $\R\times(\R\times V)$ in Section \ref{sub:formal_def}.  

\begin{proof}[Proof of Theorem \ref{t:vankampen}]
Consider the manifold $W\#_{V_{0}} W'$ and note that the core disks $l$ of the top handles of $h$ in $W$ and $W'$ joined by $\partial l\times [-1,1]$ form Legendrian attaching spheres $\Sigma_{\#}(h)$ in $\partial W\#_{V_{0}}W'$ for $W\#_{V} W'$. In other words, we obtain $W\#_{V} W'$ by attaching Weinstein $n$-handles to $W\#_{V_{0}}W'$ along $\Sigma(h)\subset \partial W\#_{V_{0}}W'$.

Lemma \ref{l:nocross} shows that holomorphic disks that contribute to the differential in 
$$ CE^{\ast}(\Sigma(h);W\#_{V_{0}}W')$$ 
of a Reeb chords of $\Sigma(h)$ in $\partial W$ or $\partial W'$ stays in $W\cup (D^{\ast}_{\epsilon}[-1,0]\times V_{0})$ and $W'\cup(D^{\ast}_{\epsilon}[0,1]\times V_{0})$, respectively, and that disks with positive puncture at a Reeb chord inside $D^{\ast}_{\epsilon}[-1,1]\times \partial V_{0}$ stays over $(0,0)\times V_{0}$ if $n>2$ and stays arbitrarily close to $\{y=0\}$ if $n=2$, see Remark \ref{r:dimV_0=2}. The theorem follows.  
\end{proof}
\subsection{Stop removal}
In this section we consider the operation of removing a stop. Consider as usual a Legendrian embedding $(V,h)\subset\partial W$ with $m$ top-dimensional handles $h_{j}$ with core disks $l_{j}$,  $j=1,\dots,m$.

\begin{definition}\label{dfn:loose_Legendrian_embedded_V}
	 The Legendrian embedding of $(V,h)$ is \emph{loose} if the core disks $l_{j}$ of all the top handles $h_{j}$, $j=1,\dots,m$ admit disjoint loose charts in $\partial W$, see \cite{EM}.
\end{definition}

As expected the Chekanov--Eliashberg dg-algebra of a loose embedding of $(V,h)$ is trivial:
\begin{lemma}\label{lma:V_loose_CE_trivial}
	Let $(V,h)\subset\partial W$ be a loose Legendrian embedding then $CE^\ast((V,h);W)$ is trivial.
\end{lemma}
\begin{proof}
	The dg-algebra $CE^{\ast}((V,h);W)$ is defined as the dg-algebra of the link $\Sigma(h)\subset \partial W_{V}^{0}$ where each component $\Sigma(h_{j})$ contains the core disk $l_{j}$ of $h_{j}$. The loose charts for $l_{j}$ give loose charts for $\Sigma(h)$ which then is loose and hence its dg-algebra is quasi-isomorphic to the trivial algebra. 
\end{proof}

Let $V$ be a Weinstein $(2n-2)$-manifold and consider the product $X_{1}=V\times D^{\ast}_{1}[-1,1]$ with rounded corners and Liouville vector field $z + x \partial_x + y \partial_y$ where $z$ is a Liouville vector field on $V$, see Figure \ref{fig:liouville}. Then $V\times \{(-1,0)\}$ is a Legendrian embedding of $V$ in $X_{1}$. 
\begin{figure}[!htb]
	\includegraphics{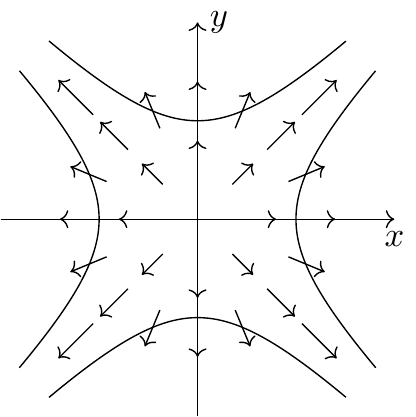}
	\caption{The projection of the Liouville vector field on the Weinstein manifold $X_1$ to the factor $D^\ast_1[-1,1]$.}
	\label{fig:liouville}
\end{figure}
\begin{lemma}\label{lma:V_loose}
	For any handle decomposition $h$ of $V$, the Legendrian embedding of $(V,h)$ corresponding to $V\times\{(-1,0)\}$ is loose in $X_{1}$. 
\end{lemma}
\begin{proof}
	We round corners as in \cite[Section 2.5.1]{AGEN3}, see Figure \ref{fig:liouville} for a comparison with the usual handle. Then applying \cite[Proposition 2.8]{CM} (compare \cite[Remark 2.23]{ENg}), we get a loose chart for each top handle.
\end{proof}

We use the lemmas above to describe the effect of adding $X_{1}$ to a Weinstein manifold stopped at $V$. Geometrically, it is clear, by canceling critical points that this corresponds to removing the stop. Here we state two results that are simple consequences of that but which are sometimes useful in calculations.

\begin{corollary}\label{cor:CEconnectsumtrivial}
	Let $X_1$ be as in Lemma~\ref{lma:V_loose} and let $(V,h)\subset \partial W$ then $CW^\ast(C(h); W \#_V X_1)$ is quasi-isomorphic to the trivial algebra.
\end{corollary}
\begin{proof}
As in the proof of Lemma \ref{lma:V_loose_CE_trivial} the link of attaching spheres $\Sigma(h)$ of the core disks dual to the co-cores in $C(h)$ is loose. The result then follows from Lemma \ref{lma:V_loose_CE_trivial} combined with Theorem \ref{t:basic} and Theorem \ref{t:vankampen}, see also Remark \ref{r:GPS}.
\end{proof}

Assume now that $W$ is a Weinstein manifold with co-core disks $C'$ dual to its critical handles with Legendrian boundary $\partial C'\subset\partial W$. Let $(V,h)\subset \partial W$ be a Legendrian embedding such that $V\cap\partial C'=\varnothing$. Let $W_{V}$ denote $W$ stopped at $V$, then we can remove the stop by forming $W=W_{V}\# X_{1}$.  

\begin{corollary}
	Let $X_1$ be as in Lemma~\ref{lma:V_loose}. Then there is a quasi-isomorphism
	\[
	CW^\ast(C' \cup C(h); W \#_V X_1) \ \approx \ CW^\ast(C'; W).
	\]
\end{corollary}

\begin{proof}
	Follows from Corollary \ref{cor:CEconnectsumtrivial}, Theorem \ref{t:vankampen} and Theorem \ref{t:basic}, see also Remark \ref{r:GPS}.
\end{proof}

\section{Generalizations}
In this section we discuss some natural generalizations of the results in earlier sections. 
\subsection{Legendrian submanifolds with boundary}\label{sec:Legwbdry}
Consider a Legendrian embedding $(V,h)\subset \partial W$ as above. Let $\Lambda$ be a $(n-1)$-manifold with boundary $\partial\Lambda$. Consider a Legendrian embedding $\Lambda\to\partial W$ such that $\partial\Lambda\subset\partial V$ is a Legendrian embedding as well. We construct a non-compact Legendrian $\Lambda\subset \partial W_{V}$ with ideal boundary $\partial\Lambda\subset 0\times \partial V\subset\R\times V$ by adding $\partial\Lambda\times [-1,1]$ and the Legendrian lift of the positive cone on $\partial\Lambda$ in $\R\times\R\times \partial V$. We define the Chekanov--Eliashberg dg-algebra
\[ 
CE^{\ast}(\Lambda;V;W)
\]
as the ordinary Chekanov--Eliashberg dg-algebra of $\Lambda\subset \partial W_{V}$. 

\begin{lemma}\label{l:Legsubalg}
The dg-algebra $CE^{\ast}(\Lambda;V;W)$ contains a subalgebra canonically isomorphic to $CE^{\ast}(\partial\Lambda; V)$.
\end{lemma}

\begin{proof}
As before, the sub-algebra is generated by Reeb chords in the middle of the handle. The proof is a repetition of the proof of Lemma \ref{l:middlesubcrit}. 
\end{proof}

\begin{figure}[!htb]
	\centering
	\includegraphics{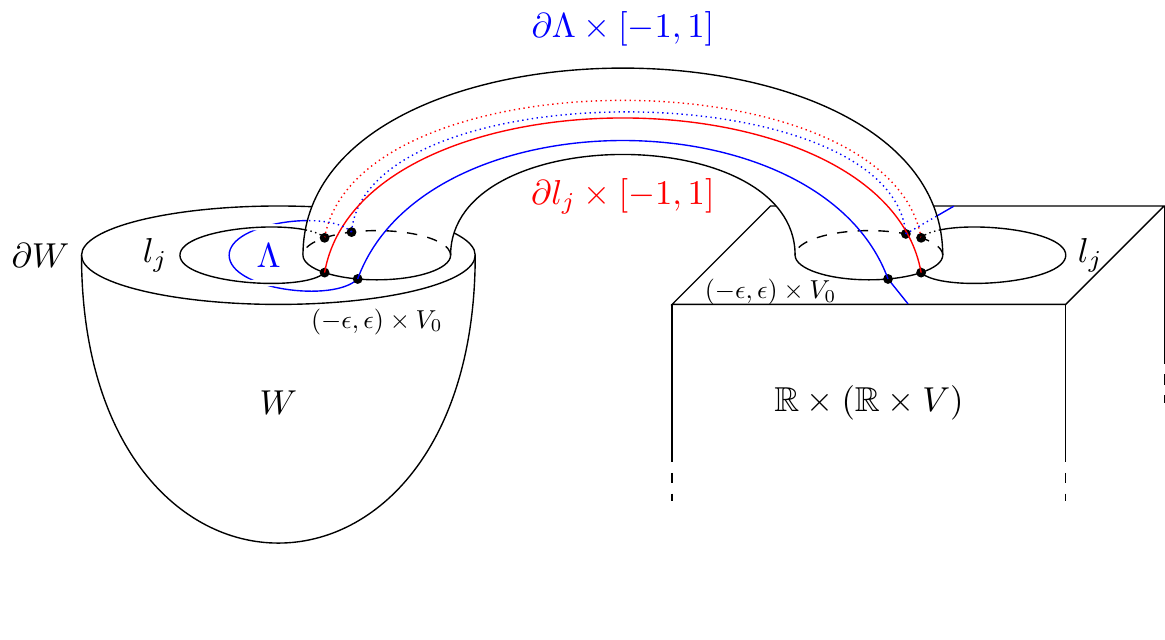}
	\caption{The non-compact Legendrian $\Lambda \subset \partial W_V$ in the definition of $CE^\ast(\Lambda; V ; W)$.}
\end{figure}

We next give a surgery description of $CE^{\ast}(\Lambda;V;W)$. By general position, we may
assume that $\partial\Lambda\subset \partial V$ is disjoint from the boundary $\partial c$ of co-core disks of the top handles in $h$. This last condition means that we can consider $\partial\Lambda\subset \partial V_{0}$, disjoint from the attaching spheres $\partial l$ of the top handles in $V$. This in turn means that we can compute the dg-algebra by Legendrian surgery. We write $\Lambda\to l$ to denote either a Reeb chord from $\Lambda$ to $l$ in $\partial W$ or a Reeb chord from $\partial\Lambda$ to $\partial l$ in $\partial V_{0}$ and use the notation $l\to l$, $l\to\Lambda$, and $\Lambda\to\Lambda$ similarly. 

\begin{proposition}\label{p:surgeryforLeg}
In terms of data in $W$ and $V_{0}$, the dg-algebra $CE^{\ast}(\Lambda;V;W)$ admits the following description. It is generated by composable words of Reeb chords of the form $\Lambda\to l\to l\to\dots\to l \to\Lambda$. The differential is induced by the Legendrian dg-algebra differential of $\Lambda\cup\Sigma(h)\subset \partial_{+} W_{V}^{0}$ and the subalgebra of Lemma \ref{l:Legsubalg} is generated by composable words of Reeb chords $\partial \Lambda\to\partial l\to\dots\to\partial l\to\partial\Lambda$ in $\partial V_{0}$.
 \end{proposition}
\begin{proof}
	This follows from \cite[Theorem 5.10]{BEE} by an argument directly analogous to the proof of Theorem \ref{t:basic}. 
\end{proof}

As in the absolute case we can join Legendrians with the same boundary in two copies of $V$ in different manifolds. More precisely, in the situation of Theorem \ref{t:vankampen}, if $W$ and $W'$ contains Legendrians $\Lambda\subset \partial W$ and $\Lambda'\subset\partial W'$ with common boundary $\partial\Lambda=\partial \Lambda'\subset \partial V$. Then we can join $\Lambda$ and $\Lambda'$ via $\partial\Lambda\times [-1,1]$ across the $V$-handle $D^{\ast}_{\epsilon} V$ to form the closed Legendrian submanifold $\Lambda\#_{\partial V}\Lambda'$. 
We next give a surgery description of the dg-algebra $CE^{\ast}(\Lambda\#_{\partial V}\Lambda',W\#_{V} W')$. As above, by general position $\partial\Lambda\subset\partial V_{0}$ and $\Lambda\#_{\partial V}\Lambda'$ is obtained from the Legendrian $\Lambda\#_{\partial V_{0}}\Lambda\subset \partial_{+}W_{V}^{0}$ after surgery on $\Sigma_{\#}(h)$. We use notation analogous to that in Proposition \ref{p:surgeryforLeg}.

\begin{proposition}\label{p:Legsurgeryconnsum}
	In terms of data in $W$, $W'$, and $V_{0}$, the dg-algebra $CE^{\ast}(\Lambda\#_{\partial V}\Lambda';W\#_{V} W')$ admits the following description. It is generated by composable words of Reeb chords of the form $\underline{\Lambda}\to l\to l\to\dots\to l \to\underline{\Lambda}$, where $\underline{\Lambda}$ denotes $\Lambda$ or $\Lambda'$. The differential is induced by the Legendrian dg-algebra differential of $(\Lambda\#_{\partial V_{0}}\Lambda')\cup  \Sigma(h)  \subset \partial_{+} W_{V}^{0}$. Furthermore, the subalgebra generated by composable words of Reeb chords $\partial \Lambda\to\partial l\to\dots\to\partial l\to\partial\Lambda$ in $\partial V_{0}$ is isomorphic to $CE^{\ast}(\partial\Lambda;V)$.
\end{proposition}
\begin{proof}
	This again follows from \cite[Theorem 5.10]{BEE}. 
\end{proof}


\subsection{Floer cohomology of singular exact Lagrangian fillings and cobordisms}\label{sec:cobordism}
In this section we outline a construction of the counterpart of `cobordism maps' for usual dg-algebras in the setting of singular Legendrians, induced by corresponding singular Lagrangians. A more complete treatment of the subject will appear elsewhere. We refer to Section \ref{sssec:singlag} for illustrations of the construction with detailed calculations in concrete examples.

Consider a Weinstein manifold $X$ with contactization $\R\times X$. We say that an embedding of a Weinstein domain $V$ with handle decomposition $h$, $(V,h)\to X$ is \emph{exact} if it lifts to a Legendrian embedding into the contactization. Note that such a lift, when restricted to the skeleton, is determined up to a translation in $\R$, one for each connected component of $V$.

\subsubsection{Floer cohomology}\label{sssec:Floercohomology}
To motivate our next definition we consider smooth exact Lagrangians $L_{0},\dots,L_{k}$ in $X$. Then the operation $\mathfrak{m}_{k}$ on the Floer cohomology complex $CF^{\ast}(L_{0}\cup\dots\cup L_{k};X)$ that counts holomorphic disks with $k$ inputs and one output, can be defined as the dual of the differential in the dg-algebra of a Legendrian link $\tilde L$. Here $\tilde L$ is obtained by shifting the Legendrian lifts of $L_{j}$ so that $L_{j}$ sits above $L_{k}$ for all $k\le j$. The Floer cohomology complex $CF^{\ast}(L;X)$ of any Lagrangian $L\subset X$ carries the structure of an $A_{\infty}$-algebra. We use parallel copies of $L$ as in \cite{EL} to define operations and in this way the Legendrian lift approach works in general for exact Lagrangian submanifolds.

We consider exact embeddings $(V,h)\subset X$ as singular Lagrangian embeddings of the skeleton of $(V,h)$ and \emph{define} the Floer cohomology $CF^{\ast}((V,h);X)$ as the $A_{\infty}$-algebra with differential and operations given by the duals of the differential in $CE^{\ast}((\tilde V,\tilde h),\R\times(\R\times X))$ for lifts $(\tilde V,\tilde h)$ into $\R\times X$, where the lifts are related exactly as for smooth Lagrangians. This then leads to new ways of representing objects in the Fukaya category of $X$, see \cite{DRET} for immersed Lagrangian 2-spheres treated in this way.

\subsubsection{Energy filtration, fillings, and cobordisms}\label{sssec:cobordisms}
We next consider the case when the singular Lagrangian is allowed to have boundary at infinity. Again we start in the smooth case. Let $L\subset X$ be an exact Lagrangian submanifold with Legendrian boundary $\partial L$. Then there is a natural map 
\[ 
CE^{\ast}(\partial L;X)\longrightarrow (CF^{\ast}(L;X))',
\]
where right hand side is the (completed) co-algebra dual to the Floer cohomology algebra, see \cite{EL}, which is defined by counting curves with one positive puncture at a Reeb chord and several negative punctures at intersection points of a system of parallel copies of $L$. Here it is natural to replace the dual in the right hand side by the Chekanov--Eliashberg dg-algebra of a Legendrian lift of $L$ into $\R\times X$, as in Section \ref{sssec:Floercohomology}, and we get the corresponding map 
\begin{equation}\label{eq:Floermapdgaver} 
CE^{\ast}(\partial L;X)\longrightarrow CE^{\ast}(L;\R\times(\R\times X)).
\end{equation}
The projection of this chain map to the subalgebra generated by the unit in $L$ is called an \emph{augmentation}. 

In case $\partial L$ is a sphere we may think of the augmentation in terms of Legendrian surgery as follows. Let the co-core of the handle attached be $C$ and let $\widehat L$ be the Lagrangian $L$ with the core disk attached. Then the Floer complex $CF^{\ast}(C,L)$ has only one generator, the unique intersection point between the core and the co-core. However, the Floer complex $CF^{\ast}(C,\widehat L)$ is a module over $CW^{\ast}(C)$, where the module structure is obtained by counting disks with several input (positive) punctures at generators of $CW^{\ast}(C)$, one input generator of $CF^{\ast}(C,\widehat L)$, and one output generator of $CF^{\ast}(C,\widehat L)$. By SFT-stretching it follows that the surgery isomorphism $CW^{\ast}(C)\to CE^{\ast}(\partial L;X)$ takes the module structure to the augmentation.

The cobordism map on the form \eqref{eq:Floermapdgaver}, generalizes to the singular case. Consider a Weinstein domain $K\subset X$ such that in the ideal boundary of $X$, $K$ agrees with $V^{+}\times ((-\epsilon,\epsilon)\times\R)$, or in terms of a compact model of $K$ and $X$, there is some region near the boundary of $X$ where $K$ agrees with $V\times D^{\ast}_{\epsilon}(-\eta,0]$, where $V\times D^{\ast}_{\epsilon}(-\eta,0]|_{\{0\}}$ lies in $\partial X$. 

We consider a handle decomposition $H$ of $K$ with, except for standard handles, also has handles with boundary in the handle decomposition $h^{+}$ of $V^{+}$. Here, the $k$-dimensional core $\Delta$ of a handle with boundary from $H$ has boundary given by $(k-1)$-dimensional core disks $\delta^{+}$ in handles from $h$. In the compact model above it means that there is a neighborhood of the boundary of $(\eta,0]$ where $\Delta$ is a product, $\delta^{+}\times(-\eta,0]$. 

With such a handle structure we can then lift $K$ to a Legendrian embedding into $\R\times X$, and add a $K$-handle $K\times D^{\ast}_{\epsilon}[-1,1]$ to $X\times\R$. We need to explain what this looks like at 
the boundary. To this end we consider two intervals $[-1,1]'$ and $[-1,1]''$. At the boundary we add the 
$V$-handle $V\times D^{\ast}_{\epsilon}[-1,1]'$, where the fiber at $-1\in [-1,1]$ 
is attached at $0\in(-\eta,0]$ and where we use the standard Liouville vector field in the handle with a saddle point at $(0,0)$. The $K$-handle now looks like $(V\times D^{\ast}_{\epsilon}[-1,1]')\times D^{\ast}_{\epsilon}[-1,1]''$, where $D^{\ast}_{\epsilon}[-1,1]''$ correspond to the Reeb and symplectization directions in $X\times \R\times\R$. We start from the Liouville vector field that is
radial along fibers in $D^{\ast}_{\epsilon}[-1,1]''$. This gives a Bott situation with a family of Reeb chords and orbits along the $0$-section. We impose the boundary condition that these are perturbed out by a Liouville vector field that point into the handle at the point $-1\in[-1,1]$ where it is attached. Imposing this boundary condition at all handles we can extend by half infinite lines and consider the boundary of the handles to form an ideal boundary of $(X,K)$.  

Note that using this model, the attaching locus for the top-dimensional handle cores $L$ in $H$ is a Legendrian cobordism $\partial L$ which near the boundary looks like the product of the $0$-section in $D^{\ast}_{\epsilon}[-1,1]''$ and $\partial l^{+}$, and hence after the perturbation has Legendrian ends in the attaching spheres $\partial l^{+}$, where $l^{+}$ are the top-dimensional core disks in $V$ and $\partial l^{+}$ their attaching spheres in $\partial V_{0}$.

We define the dg-algebra $CE^{\ast}((K,H);\R\times(\R\times X))$ in parallel with the dg-algebra above. It is generated by Reeb chords of the top handles $L$ in $H$ inside $\R\times X$ and Reeb chords of the Legendrian attaching cobordisms $\partial L$ of these handles. In this setting we think of the Chekanov--Eliashberg dg-algebra $CE^{\ast}((V^{+},h^{+});X)$ as generated by Reeb chords in $\partial X$ and of Reeb chords inside the $V_{0}$-handle as sitting at the ideal boundary of the positive symplectization end of $X$ and of $\partial K_{0}$, respectively. Counting disks with one positive puncture at these chords at infinity then gives a chain map
\[ 
CE^{\ast}((V^{+},h^{+}); X)\longrightarrow CE^{\ast}((K,H);\R\times(\R\times X)).
\]  
By shrinking the handle we find that the sub-algebra $CE^{\ast}(\partial l^{+},\partial V_{0})$ maps to the sub-algebra $CE^{\ast}(\partial L;\partial K_{0})$. 

Also the discussion about augmentations have counterparts in this set up. Here we consider the Floer cohomology with the fiber of the singular Lagrangian. Consider the case that there are no Reeb chords of $L$ in $X\times\R$ then the counterpart of the augmentation is the dg-algebra map above followed by the projection to the subalgebra $CE^{\ast}(\partial L;\partial K_{0})$ which gives the module structure of the Floer cohomology $CF^{\ast}(C,K)$ that must now be viewed as a Floer cohomology 'with coefficients' in the dg-algebra of the link of the singularity of $K$. As the Floer cohomology itself is again very simple, the augmentation naturally takes values in the dg-algebra of this link of singularities, see Section \ref{sec:exandappl} for examples. 

\begin{figure}[!htb]
	\includegraphics{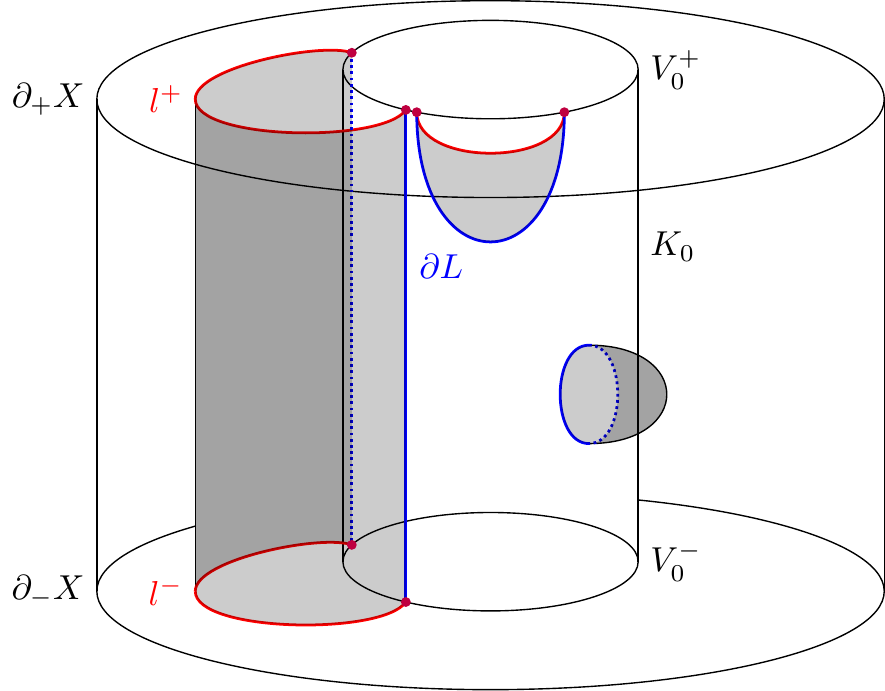}
	\caption{The subcritical part of the Weinstein domain $K$ in the Weinstein cobordism $X$. The cores of the top handles $L$ are shaded in gray, and its boundary in $\partial K_0$ is drawn in blue.}
\end{figure}

One can also extend the discussion here and allow the `cobordism' $(K,H)$ to have a negative end as well. We first take $X$ to have a negative end $\partial_{-}X$ that we assume is filled by a Weinstein manifold $X_{-}$ and in $\partial_{-}X$ we have a Legendrian embedding $(V^{-},h^{-})$. We then require that $(K,H)$ agrees with $V^{-}\times ((-\epsilon,\epsilon)\times\R)$ near the negative end and that the handle decomposition $H$ of $K$ is allowed to have handles with boundary. With such a handle structure we can then again lift $K$ to a Legendrian embedding with a cylindrical end in $\R\times X$. The only difference from the treatment above is that we perturb the Bott family over $D^{\ast}_{\epsilon}[-1,1]''$ with a Liouville vector field that points out of the cobordism in the negative end.

The attaching locus for a top-dimensional handle is then a Legendrian cobordism $\partial L$ with Legendrian ends in the attaching spheres $\partial l^{+}$ at the positive end and in $\partial l^{-}$ in the negative end.

Here we define the dg-algebra $CE^{\ast}((K,H);\R\times(\R\times X))$ in parallel with the above as generated by Reeb chords of the top handles $L$ in $\R\times X$, Reeb chords of the Legendrian attaching cobordisms $\partial L$, as well as Reeb chords between core disks $l^{-}$ and their attaching spheres $\partial l^{-}$ in the negative end. Here the two types of Reeb chords in the negative end form a subalgebra isomorphic to $CE^{\ast}((V^{-},h^{-});X_{-})$.

We again think of the Chekanov--Eliashberg dg-algebra of $(V^{+},h^{+})$ as generated by Reeb chords of $l^{+}$ in $\partial X$ and Reeb chords of $\partial l^{+}$ as sitting at the ideal boundary of the positive symplectization end of $X$ and of $\partial K_{0}$, respectively. Counting disks with one positive puncture at these chords at infinity then again gives a chain map
\[ 
CE^{\ast}((V^{+},h^{+}); X)\longrightarrow CE^{\ast}((K,H);\R\times(\R\times X)),
\]  
where there are now chords and disks at the negative end in the right hand side.

\section{Computations, examples, and applications}\label{sec:exandappl}
In this section we first describe a method for computing the differential in Chekanov--Eliashberg dg-algebras when there is a global Reeb-projection. We then study a number of examples including non-existence results for Lagrangian fillings with restricted singularities and illustrations of how dg-algebras of singular Legendrians contain information of nearby smooth Legendrians. 
\subsection{The differential for contactizations}\label{ssec:contactizations}
In this section we consider a useful way to compute $CE^{\ast}((V,h);W)$ in the case that $\partial W$ is a contactization. Consider thus the case when $W=\R\times(\R\times P)$, where $P$ is an exact symplectic manifold, $\R\times P$ its contactization and $\R\times(\R\times P)$ the further symplectization. Note that there are no Reeb orbits in $\R\times P$ which allows us to work without anchoring.

Assume that $V\subset \R\times P$ is a generic Legendrian embedding so that $\pi|_{V_{0}}$, where $\pi$ is the projection projecting out $\R$, is an embedding. Consider then a new contact manifold $\R\times P^{\circ}$, where $P^{\circ}$ is obtained from $P$ by removing $V_{0}$ leaving the negative contact boundary $\partial V_{0}$ and inserting in its place the negative end $[0,-\infty)\times \partial V_{0}$. Another way to think about this contact manifold is the manifold that results from $\R\times P$ if all Reeb flow lines through the skeleton on $V_{0}$ is removed. 

The projection of the core $(n-1)$-disks $l$ of the top handles in $h$ to $P^{\circ}$ is then an immersed exact Lagrangian $\bar l$ with negative end $\partial l$ in $\partial V_{0}$. By construction, the Legendrian lift $l\subset \R\times P^{\circ}$ of the exact Lagrangian $\bar l$ has constant $z$-coordinate of equal value on all the components of $\partial l$ in the negative end. We can then define the Legendrian dg-algebra $CE^{\ast}(l,\R\times P^{\circ})$ in the standard way, with generators Reeb chords of double points of $\bar l$ and Reeb chords of $\partial l$ in the negative end, compare \cite{CD-RGG}. Recall that we use an index definite Weinstein structure on $V_{0}$. This means anchoring is trivial if $\dim(V_{0})\ge 4$ and if $\dim(V_{0})=2$ we anchor at Reeb orbits in the negative end $\partial V_{0}$ using punctured spheres in the filling $V_{0}\approx \R^{2}$, see Remarks \ref{r:dimV_0=2} and \ref{r:dimV_0=2discussion}. 

\begin{lemma}\label{l:diffincontact}
	The natural identification of Reeb chord generators gives a chain isomorphism
	\[ 
	CE^{\ast}(l;\R\times P^{\circ})\stackrel{\approx}{\longrightarrow}CE^{\ast}((V,h);\R\times(\R\times P)) .
	\] 
\end{lemma}
\begin{proof}
Using a complex structure on the $V_{0}$-handle as in Section \ref{ssec:holcurveinhandle} it is straightforward to check that curves contributing to the differential in the left and right hand sides above can be identified: 

First, it follows from Lemma \ref{l:complexproj} that holomorphic disks with negative punctures at the chords in the handle must lie in ${V_{0}}_{(0,y)}$-fibers, i.e., fibers over $(0,y)$ in the model of the symplectization of $G_{\delta}$, which have the form $\partial V_{0}\times\R$. (Here the $\R$-translation corresponds simply to moving the curve in the fiber along the $y$-axis.) Second, curves in the symplectization of $\R\times P^{\circ}$ are determined up to translation by their projection to $P^{\circ}$. Consequently, the $J_{V}$-biholomorphic map from $\partial V_{0}\times [0,\infty)\subset P^{\circ}$ to the $V_{0}$-fibers over the $y$-axis relates the curves in question. The lemma follows.  
\end{proof}
\begin{remark}
	Taking $P=\R^{2}$ in Lemma~\ref{l:diffincontact}, we find that our Chekanov--Eliashberg dg-algebra is isomorphic to the Chekanov--Eliashberg dg-algebra defined for Legendrian graphs by An--Bae in \cite{AB}, see in particular \cite[Theorem F]{AB}.
\end{remark}

\subsection{Examples}
In this section we study Chekanov--Eliashberg dg-algebras for singular Legendrians in several examples.

\subsubsection{The $n$-point algebra and $T^{\ast}\R^{2}$ from the point of view of singular Legendrians}
We study Chekanov--Eliashberg dg-algebras in dimension 1 and consider the case left out in Section \ref{ssec:cotangentS^n}.

\begin{example}\label{ex:3pts}
	We discuss the relation between wrapped Floer cohomology and Chekanov--Eliashberg dg-algebras for Weinstein handles attached to the 2-disk, following \cite{EL2}. The dg-algebra involved will appear as 'singularity link dg-algebra' in several examples below (also see Figure~\ref{fig:A2_sing}).
	Consider $(\R^{2},dx \wedge dy)$ with ideal contact boundary $(S^1,\frac12(xdy-ydx))$. Let $\Lambda$ be $n$ distinct points in the ideal contact boundary and let $V=T^\ast \Lambda \subset S^1$. Since $V$ is zero-dimensional, the only generators of $CE^\ast((V,h);\R^{2})$ are Reeb chords in $S^1$ of the core disks of the top handles $l = \Lambda$, and these are the following:
	\begin{itemize}
		\item $c^0_{ij}$ for $1 \leq i < j \leq n$,
		\item $c^p_{ij}$ for $1 \leq i,j \leq n$ and $p \geq 1$,
	\end{itemize}
	where the generator $c^p_{ij}$ is the Reeb chord starting at the $i^{\rm th}$ point, ending at the $j^{\rm th}$, and passing through the reference point $\ast$ $p$ times. 
	\begin{figure}[!htb]
		\centering
		\includegraphics{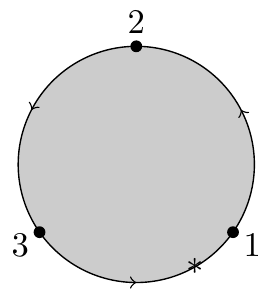}
		\hspace{2cm}
		\raisebox{3mm}{\includegraphics[scale=1.2]{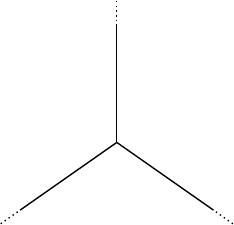}}
		\caption{Left: Generator Reeb chords $c^0_{12}$, $c^0_{23}$, and $c^1_{31}$ of $CE^\ast((V,h);\R^{2})$ when $\Lambda$ consists of three distinct points. Right: The arboreal $A_2$-Lagrangian in $\R^2$.}
		\label{fig:A2_sing}
	\end{figure}
	
	Choosing Maslov potentials $(m(1),\ldots,m(n))$ for the points in $\Lambda$ we get the following grading of the generators:
	\begin{equation}\label{eq:gradingI_n}
		|c^p_{ij}| = 1-2p - m(j)+m(i)\, .
	\end{equation}
	The differential $\partial$ is given by
	\begin{align}
		\partial(c^0_{ij}) &= \sum_{k=1}^n (-1)^{m(i) + m(k)}c^0_{kj}c^{0}_{ik} \label{eq:d1} \\
		\partial(c^1_{ij}) &= \delta_{ij} + \sum_{k=1}^n (-1)^{m(i)+m(k)} c^1_{kj}c^0_{ik} + \sum_{k=1}^n (-1)^{m(i)+m(k)} c^0_{kj}c^1_{ik} \label{eq:d2} \\
		\partial(c^p_{ij}) &= \sum_{\ell = 0}^p \sum_{k=1}^n (-1)^{m(i)+m(k)} c^{p- \ell}_{kj} c^{\ell}_{ik}, \quad p\geq 2 \label{eq:d3}\, ,
	\end{align}
	where $\delta_{ij} = e_i = e_j$ when $i = j$ and $\delta_{ij} = 0$ otherwise. (The idempotents come from disks anchored in the $\R^{2}$-filling.) In these formulas we use the convention $c^0_{ij} = 0$ for $i\geq j$. As usual, the differential extends to all of $CE^\ast((V,h);\C)$ by Leibniz rule. This means that $CE^\ast((V,h);\R^{2})$ is the internal algebra $\mathscr{I}_{n}$ of \cite[Definition 8]{EL2}.
	
	\begin{figure}[!htb]
		\centering
		\includegraphics{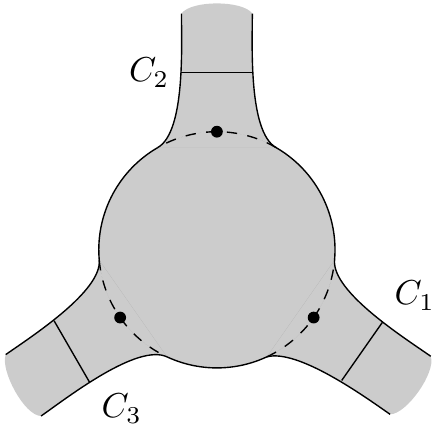}
		\caption{$\R^{2}_V$ with co-core disks $C = C_1 \cup C_2 \cup C_3$ when $\Lambda$ is three distinct points.}
	\end{figure}
	Add a stop at $V$ and let $C = C_1 \cup \cdots \cup C_n$ be the union of the co-core disks dual to the top handles in $V$. Then $CW^\ast(C;\R^{2}_V)$ is generated by $n$ Reeb chords $\{c_{12}, \ldots, c_{(n-1)n}, c_{n1}\}$ in $\partial \R_V$, where $c_{ij}$ denotes the unique Reeb chord starting at $\partial C_i$ and ending at $\partial C_j$. After a choice of Maslov potentials $(m(1),\ldots,m(n))$ of the components of $C$ we get the following grading: 
	\begin{align*}
		|c_{i(i+1)}| &= 1 + m(i+1)-m(i), \quad 1 \leq i \leq n-1, \\
		|c_{n1}| &= -1 + m(1)-m(n).
	\end{align*}
	Non-vanishing $A_\infty$-operations have as input any cyclic permutation of the cyclic sequence of Reeb chords $(c_{n1}, c_{(n-1)n},\ldots,c_{12})$. That is, for any $1\leq i \leq n-1$ we have
	\[
		\mathfrak m_n(c_{(i-1)i}, c_{(i-2)(i-1)}, \ldots, c_{n1}, c_{(n-1)n},\ldots,c_{(i+1)(i+2)},c_{i(i+1)}) = e_i\, .
	\]
	Then \cite[Corollary 11]{EL2}, shows that there is a quasi-isomorphism 
	$CW^\ast(C;\R^{2}_V) \approx \mathscr{I}_{n} = CE^\ast((V,h);\R^{2})$. 
\end{example}

\begin{example}\label{ex:unknot}
	We study the remaining  $n=2$ case of the example Section \ref{ssec:cotangentS^n} .
	Consider $(\R^4,dx_1 \wedge dy_1 + dx_2 \wedge dy_2)$, with ideal contact boundary $(S^3,\frac12(x_{1}dy_{1}-y_{1}dx_{1}+x_{2}dy_{2}-y_{2}dx_{2}))$.
	\begin{figure}[!htb]
		\centering
		\includegraphics{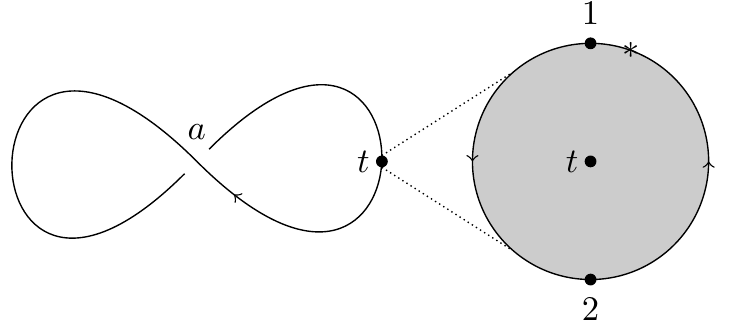}
		\caption{Lagrangian projection of the unknot in a Darboux chart and the handle decomposition of its cotangent neighborhood.}
		\label{f:unknot}
	\end{figure}
	Consider the unknot $\Lambda = \partial \R^2 \subset S^3$, and let $V\subset S^{3}$ be a small neighborhood of the zero-section in $T^\ast\Lambda$. The handle decomposition of $V$ is given by $h_0 = \{N(t)\}$, and the 1-handle $h_{1}$ is a cotangent neighborhood of the knot itself after removing the point $t$ as in Figure \ref{f:unknot}. The right hand part of Figure \ref{f:unknot} shows $V_{0}=h_0$ with $\partial l$, which consists of the points denoted $1$ and $2$ in $\partial V_{0}$. Let $l$ denote the core disk of the top handle $h_1$.

	The Chekanov--Eliashberg dg-algebra of $\partial l \subset \partial V_0$ is computed as in Example \ref{ex:3pts} and is generated by $t^0_{12}$ and $t^p_{ij}$ for $p\geq 1$ and $1 \leq i,j \leq 2$. After choosing Maslov potential $(m(1),m(2)) = (1,0)$ their gradings are $|t^p_{ij}| = 1-2p + m(j)-m(i)$.

	The Chekanov--Eliashberg dg-algebra $CE^\ast((V,h);\R^4)$ is generated by the generators of $CE^\ast(\partial l; V_{0})$ as described above and one additional generator $a$ of degree $-1$. By Lemma \ref{l:diffincontact}, the differential $\partial$ is computed via Figure \ref{f:unknot} as
	\begin{align}
		\partial a &= 1 - t^0_{12} \label{eq:1_unknot}\\
		\partial t^0_{12} &= \partial t^1_{21} = 0 \label{eq:2_unknot}\\
		\partial t^1_{ij} &= \delta_{ij} - \sum_{k=1}^2 t^1_{kj}t^0_{ik} - \sum_{k=1}^2 t^0_{kj}t^1_{ik} \qquad i \leq j \label{eq:3_unknot} \\
		\partial t^p_{ij} &= \sum_{\ell=1}^p \sum_{k=1}^2 (-1)^{m(i) + m(k)} t^{p- \ell}_{kj} t^\ell_{ik}, \qquad p\geq 2\, . \nonumber
	\end{align}
	(Here $\delta_{ij}$ is the Kronecker delta.) Note that our ground field in this case is $\C$, since $\Lambda$ is connected. As in \cite[Theorem 12]{EL2} we can consider the quasi-isomorphic model of $CE^\ast((V,h);\R^4)$ with generators $\left\{a,t^0_{12},t^1_{21}, t^1_{11}, t^1_{22}\right\}$ and differential as in \eqref{eq:1_unknot}, \eqref{eq:2_unknot} and \eqref{eq:3_unknot}. One can show that the homology is concentrated in degree 0 (see \cite[Proposition 14]{EL2}) and we compute
	\begin{equation}\label{eq:hce_ex}
		CE^0((V,h);\R^4) \approx \C[t^0_{12}, t^1_{21}]/\langle1-t^0_{12}, 1 - t^1_{21}t^0_{12}, 1 - t^0_{12}t^1_{21}\rangle \approx \C \, ,
	\end{equation}
	As in Section \ref{ssec:cotangentS^n}, $\R^4_V\approx T^{\ast}\R^{2}$ and we get the desired  quasi-isomorphisms: 
	\[
		 CE^0((V,h);\R^4) \ \approx  \ CW^\ast(C;\R^4_V) = CW^\ast(T^\ast_\xi \R^2;T^{\ast}\R^{2}) \ \approx \ C_{-\ast}(\Omega_\xi \R^2) \ \approx  \ \C .
	\]
\end{example}

\begin{example}\label{ex:noninv}
	We reconsider the Legendrian unkot of Example~\ref{ex:unknot} with a different handle structure, as an illustration of the dependence of the dg-algebra on the handle structure. We use two $0$-handles $N(t(1))$ and $N(t(2))$ and two $1$-handles $h_{1}$ and $h_{2}$ connecting them.
	\begin{figure}[!htb]
		\centering
		\includegraphics{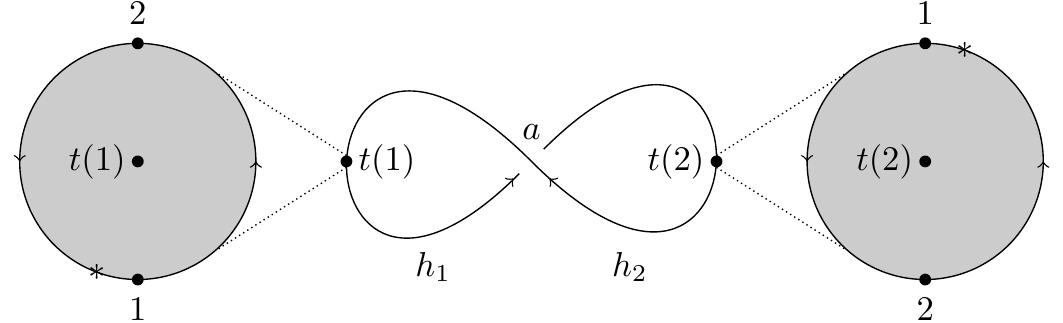}
		\caption{Lagrangian projection of the unknot in a Darboux chart with a different handle decomposition than in Example~\ref{ex:unknot}.}
		\label{f:unknot2}
	\end{figure}
	Disks contributing to the differential are shown in Figure~\ref{f:unknot2}. The differential of $a$ is given by
	\[ 
	\partial a = t(1)_{12}^{0} - t(2)_{12}^{0}
	\]
	and the differential of the generators $\{t(1)^p_{ij},t(2)^p_{ij}\}$ is given as in \eqref{eq:2_unknot} and \eqref{eq:3_unknot}. The homology is concentrated in degree $0$, and we have
	\[
		CE^0((V,h);\R^4) \approx \C[t(1)^0_{12},t(2)^0_{12},e_1,e_2]/I\, ,
	\]
	where 
	\[
		I = \langle t(1)_{12}^{0} - t(2)_{12}^{0}, e_1 - t(1)^1_{21}t(1)^0_{12},e_2 - t(1)^0_{12}t(1)^1_{21}, e_1 - t(2)^1_{21}t(2)^0_{12},e_2 - t(1)^0_{12}t(2)^1_{21}\rangle\, .
	\]
	It follows that $CE^{\ast}((V,h);\R^{4})$ is quasi-isomorphic to an algebra with two idempotents $e_{1}$ and $e_{2}$ and two generators $c_{1}$ and $c_{2}$ with relations $e_{j}c_{k}=\delta_{jk}c_k$, $c_{k}e_{j}=(1-\delta_{kj})c_k$, $c_1c_2=e_{1}$, and $c_2c_1=e_{2}$, which is the endomorphism algebra of two cotangent fibers in $T^{\ast}\R^{2}$, where $c_{j}$ are Reeb chords $c_{j}$ corresponding to straight line geodesics connecting the points where the fibers sit and where the products $c_{i}c_{j}$ correspond to disk with two positive punctures constrained by the minimum in either fiber disk. 
\end{example}

\subsubsection{Singular Lagrangian fillings in $\R^{4}$}\label{sssec:singlag}
We study dg-algebra maps of singular Lagrangian fillings in $\R^{4}$ and use them to show non-existence results for Lagrangians with restricted singularities. Throughout we write $\R^{4}$ for the standard symplectic 4-space with symplectic form $dx_{1}\wedge dy_{1}+dx_{2}\wedge dy_{2}$. We will think of $\R^{4}$ in two ways, as the symplectization of $\R^{3}$ with contact form $dz-ydx$ and as the Weinstein manifold with a single 0-handle. We will often present front pictures of Legendrians (projections to the $xz$-plane) and also for Lagrangians in $\R^{4}$ when viewed the symplectization of $\R^{3}$ where Lagrangians appear as asymptotically conical fronts, see \cite[Section 2]{EKH} for details.

\begin{example}\label{ex:singular_saddle}
	Let $\Gamma$ be the singular exact Lagrangian cobordism in $\R^{4}$ with front as in Figure~\ref{fig:birth_cob}. Let $\Lambda_+$ and $\Lambda_-$ denote the positive and negative boundaries of $\Gamma$ respectively. Let $(V^\pm,h^\pm)$ be Weinstein thickenings of $\Lambda_\pm$. Let $K \subset \R^4$ be a Weinstein domain which agrees with $(-\epsilon, \epsilon) \times V^+$ and $(-\epsilon,\epsilon) \times V^-$ in the positive and negative ideal contact boundaries of $\R^4$, see Section \ref{sssec:cobordisms}. Let $H$ be a handle decomposition of $K$ induced by $\Gamma$, and denote the union of the core disks of the critical handles of $H$ by $L$.
	\begin{figure}[!htb]
		\centering
		\includegraphics{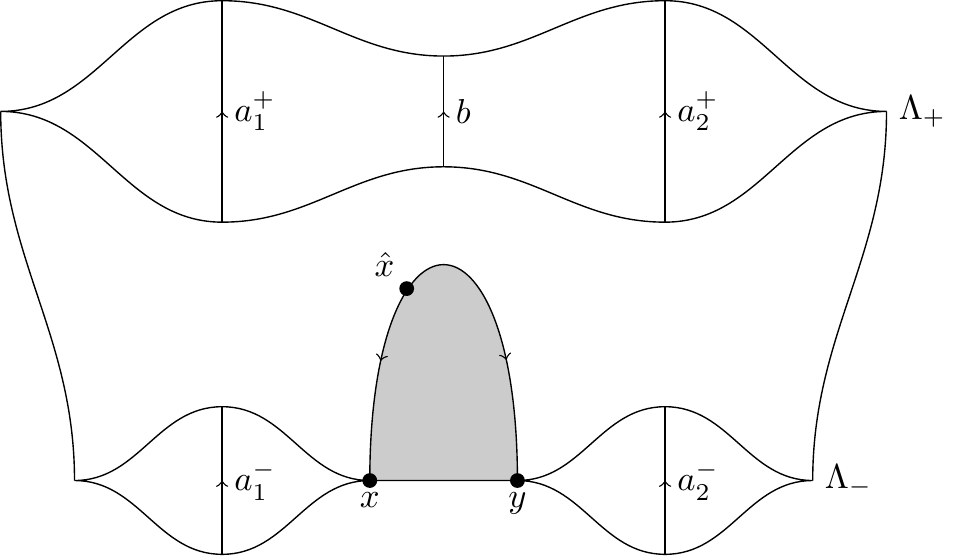}
		\caption{The front of the singular saddle cobordism $\Gamma$.}
		\label{fig:birth_cob}
	\end{figure}

	Consider $CE^\ast((K,H); \R \times (\R \times \R^4))$ as described in Section~\ref{sec:cobordism}. It contains the dg-subalgebra
	\[
		CE^\ast(\partial L; \partial K_0) \subset CE^\ast((K,H); \R \times (\R \times \R^4))\,.
	\]
	Recall that the boundary condition at the negative end of the $K$-handle means that the Liouville field points out of the handle which means there is a maximum along core of the attaching 1-handle with boundary in $\partial K_{0}$. We indicate the location of this maximum by $\hat x$ in Figure~\ref{fig:birth_cob} and denote the corresponding dg-algebra generators by $\left\{\widehat x^p_{ij}\right\}$. Together with the generators of the three point algebras in the negative end these form the Chekanov--Eliashberg dg-algebra $\widehat{\mathscr I}_{3}$ of 3 distinct points in $S^1$, multiplied by the 0-section in $T^{\ast}\R$ with a maximum and two minima in it. Apply \cite[Corollary 5.6]{EK} to Example~\ref{ex:3pts} to see the differential $d$:
	\begin{equation}\label{eq:diffhatI3}
		d \widehat x^p_{ij} = x^p_{ij} - y^p_{ij} + G(\partial \widehat x^p_{ij}),
	\end{equation}
where $\partial \widehat x^p_{ij}$ is the differential of $\widehat x^p_{ij}$ regarded as a generator of the 3-point algebra $\mathscr I_3$ and $G$ is the following operator on monomials: 
\begin{align*}
G(\widehat x^{p_{1}}_{i_{1}j_{1}}\widehat x^{p_{2}}_{i_{2}j_{2}}\dots \widehat x^{p_{m}}_{i_{m}j_{m}}) &=
\widehat x^{p_{1}}_{i_{1}j_{1}}x^{p_{2}}_{i_{2}j_{2}}\dots x^{p_{m}}_{i_{m}j_{m}}+
(-1)^{|x^{p_{1}}_{i_{1}j_{1}}|}y^{p_{1}}_{i_{1}j_{1}}\widehat x^{p_{2}}_{i_{2}j_{2}}\dots x^{p_{m}}_{i_{m}j_{m}}\\
&\quad +\dots+
(-1)^{|x^{p_{1}}_{i_{1}j_{1}}|+\dots+|x^{p_{m-1}}_{i_{m-1}j_{m-1}}|}y^{p_{1}}_{i_{1}j_{1}}y^{p_{2}}_{i_{2}j_{2}}\dots \widehat x^{p_{m}}_{i_{m}j_{m}}.
\end{align*}
Holomorphic curves in $\R \times (\R \times \R^2)$ with boundary on $\R \times K$ can be understood through Morse flow trees, which here since there are only two sheets locally are in fact flow lines. They give the cobordism dg-algebra map defined on generators as follows
	\begin{align*}
		\Phi\colon CE^\ast((V_+,h_+); \R^2) &\longrightarrow  CE^{\ast}((K,H); \R\times\R\times\R^{2}),\\
		a_1^+ &\longmapsto a_1^- + \widehat x^0_{12},  \\
		a_2^+ &\longmapsto a_2^-, \\
		b &\longmapsto y^0_{12}\,.
	\end{align*}
The disks giving the first terms in the first two equations correspond to straight flow lines. The last equation comes from the flow line starting at $b$ and hitting the singular locus and then continuing as a flow line in the singular locus to negative infinity. The second term in the first equation can be understood as coming from the 1-parameter family of disks obtained by gluing the strip at positive infinity with positive puncture at $a$ and negative puncture at $b$ to the disk from $b$ to the singular locus. This one parameter family travels into the cobordism and gets rigidified when it hits $\widehat x$. The rigidified disk then has a negative puncture at $\widehat x$. We verify algebraically that it indeed is a dg-algebra map:
	\begin{align*}
		(d \circ \Phi)(a_1^+) &= d(\widehat x^0_{12} +a_1^-) = 1+y^0_{12} = \Phi(1+b) = (\Phi \circ \partial)(a_1^+) \\
		(d \circ \Phi)(a_2^+) &= d(a_2^-) = 1+y^0_{12} = \Phi(1+b) = (\Phi \circ \partial)(a_2^+) \\
		(d \circ \Phi)(b) &= d(y^0_{12}) = (0,0) = (\Phi \circ \partial)(b)\, .
	\end{align*}
	\end{example}

\begin{example}\label{ex:death_filling_doesnt_exist}
	Let $\Lambda\subset\R^{3}$ be the singular Legendrian whose front projection is in Figure \ref{fig:unknot_edge}. Let $(V,h)$ be a Weinstein thickening of $\Lambda$. We show that $\Lambda$ does not admit any singular  Lagrangian filling with only 'Y-singularities', i.e., $\Gamma$ is a union of smooth 2-dimensional Lagrangian strata that meets along $S\times Y$, where $Y$ is a trivalent graph (see Figure~\ref{fig:A2_sing}) and $S$ is a 1-manifold with at least one component which has boundary on the trivalent $Y$ graphs around the singularities of $\Lambda$.
	
	To see this, assume that $\Gamma$ is such a filling and let $K \subset \R^{4}$ be its Weinstein thickening. Let $H$ be a handle decomposition of $K$ with critical core disks $L$ and with boundary on the core disks of $\Lambda$.
	\begin{figure}[!htb]
		\centering
		\includegraphics{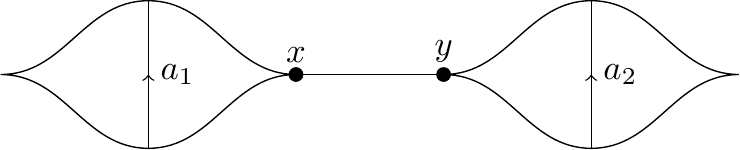}
		\caption{The front projection of the singular Legendrian $\Lambda$.}
		\label{fig:unknot_edge}
	\end{figure}
	We compute $CE^\ast((K,H); \R \times (\R \times \R^{4}))$ and show that there is no dg-algebra map $CE^{\ast}((V,h);\R^{4})\to CE^\ast((K,H); \R \times (\R \times \R^{4}))$ to conclude that no such singular exact Lagrangian filling exists. 

	Since there are no double points of $L$, $CE^\ast((K,H); \R \times (\R \times \R^4)) \approx CE^\ast(\partial L; \partial K_0)$. Since $\partial L$ is the product of the 3-point Legendrian in $S^1$ and a compact $1$-manifold $S$ we find that $CE^\ast(\partial L; \partial K_0)$ is a free product of algebras one for each component $S$ over the ring of idempotents of the $2$-cells. The algebra corresponding to a closed component is $\widehat{\mathscr{I}}_{3}$, see \eqref{eq:diffhatI3} with generators $x_{ij}^{p}$ and $y_{ij}^{p}$ identified. Here we can think of the generators as sitting at a maximum and a minimum in the circle component. The algebra corresponding to the component with boundary at infinity is $\mathscr{I}_{3}$ itself, see Example~\ref{ex:3pts}, with generators sitting at a minimum, recall the inwards boundary condition at positive infinity . We consider the map
	\[ 
	CE^{\ast}((V,h);\R^{4})\to CE^{\ast}((K,H);\R\times(\R\times\R^{4}))\to\mathscr{I}_{3},
	\]
	where the last map is the projection to the sub-algebra corresponding to the component of $S$ with boundary. We write $\{x^p_{ij}\}$ and $\{y^p_{ij}\}$ for the generators of the two copies of $\mathscr{I}_{3}$ at the singularities of $\Lambda$ and $\{\widecheck{x}^{p}_{ij}\}$ for the generators of $\mathscr{I}_{3}$ corresponding to the minimum in the non-closed component of $S$. The differential in $CE^\ast((V,h);\R^4)$ is given on the generators $a_1$ and $a_2$ by
	\begin{align*}
		\partial a_1 &= e_1 - x^0_{12} \\
		\partial a_2 &= e_2 - y^0_{12}\, .
	\end{align*}
	The differential on the generators $\{x^p_{ij}\}$ and $\{y^p_{ij}\}$ is given by \eqref{eq:d1}, \eqref{eq:d2} and \eqref{eq:d3} and has the following property. For any generator $c \in \mathscr{I}_{3}$, either $\partial c$ or $\partial c - e_i$, for some $i \in \left\{1,2,3\right\}$, is a sum of word length 2 generator monomials. It follows that for any monomial of generators $w$ of word length $k$, $\partial w = s_{k-1} + s_{k+1}$ where $s_{k\pm 1}$ are sums of monomials of word length $k\pm 1$ and $s_{k+1}$ is not the empty word. 

	The dg-algebra map $\phi$ induced by the proposed singular exact Lagrangian filling $\Gamma$ acts as follows on $\{x^p_{ij}\}$ and $\{y^p_{ij}\}$:
	\[
		\phi(x^p_{ij}) = \phi(y^{p}_{ij}) = \widecheck x^p_{ij}\,,
	\]
	and therefore
	\begin{equation}\label{eq:wrong}
		\partial(\phi(a_1)) = \phi(\partial a_1) = e_1 + \widecheck x^0_{12} \, .
	\end{equation}
	Equation \eqref{eq:wrong} contradicts $\mathscr{I}_{3}$ being non-trivial as follows. Write $\phi(a_{1})=t_{\rm e} + t_{\rm o}$, where $t_{\rm e}$ and $t_{\rm o}$ are linear combinations of monomials of even and odd word length, respectively. Since the differential changes word length mod 2 it follows that $\partial t_{\rm o}=e_{1}$ which is not true in $\mathscr{I}_{3}$. 
	
	One can also show the non-existence of $\Gamma$ from a geometrical point of view. Concatenating the singular exact Lagrangian cobordism in Example \ref{ex:singular_saddle}, which is depicted in Figure \ref{fig:birth_cob}, with $\Gamma$ gives a singular exact Lagrangian filling of the unknot. Removing top dimensional components that have no boundary at infinity, one would construct an embedded exact Lagrangian filling of the unknot of genus $\ge 1$, but such a filling does not exist.
\end{example}

\begin{example}\label{ex:theta_leg}
	Consider $\R^4$, as the completion of the symplectic ball with ideal contact boundary standard contact $S^3$. Consider the singular exact Lagrangian
	\[
		\Gamma = \{(x_1,x_2)\} \cup \left\{(y_1,x_2) \mid y_1 \geq 0\right\}\subset \R^4\, .
	\]
	We point out that if we identify $\R^4$ with $T^\ast \R^2$ then $\Gamma \approx \R^2 \cup L^+_{\left\{x_2 = 0\right\}} \subset T^\ast \R^2$ where $L^+_{\left\{x_2 = 0\right\}}$ is the positive conormal bundle of the hypersurface $\left\{x_2= 0\right\}\subset \R^2$. Thus, $\Gamma$ is an arboreal singularity, see \cite{Nadler,AGEN1,AGEN3}.

	Let $\Lambda = \Gamma \cap S^3$ be the singular Legendrian boundary of $\Gamma$. After Legendrian isotopy, $\Lambda$ lies in a Darboux ball identified with an open subset of standard contact $\R^{3}$, and $\Lambda$ appears as in Figure \ref{fig:lag_proj_theta}. 
	Consider $V = T^\ast \Lambda \subset S^3$ with handle decomposition $h$ as indicated in Figure~\ref{fig:lag_proj_theta}. Let $K \subset \R^4$ be a Weinstein domain which agrees with $(-\epsilon, \epsilon) \times V$ in the ideal contact boundary of $\R^4$. Let $H$ be a handle decomposition of $K$ induced by $\Gamma$, and denote the union of the core disks of the critical handles of $H$ by $L$.
	\begin{figure}[!htb]
		\centering
		\includegraphics{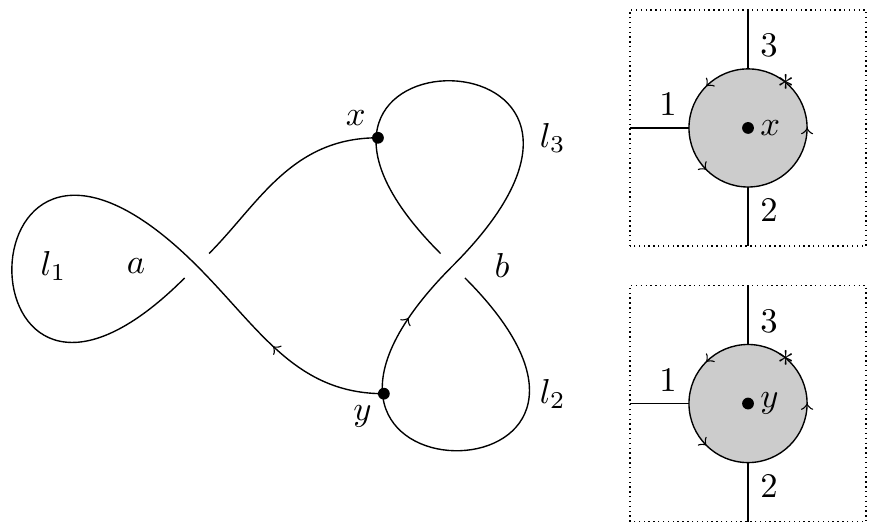}
		\caption{Lagrangian projection of $\Lambda$ in a Darboux ball (left) with magnified singularity links (right).}
		\label{fig:lag_proj_theta}
	\end{figure}
	The Chekanov--Eliashberg dg-algebra of $(V,h)$ is generated by the long Reeb chords $a$, $b$ and by the collection $\left\{x^p_{ij}\right\} \cup \left\{y^p_{ij}\right\}$ which generates the dg-subalgebra $CE^\ast(\partial l; \partial V_0)$, and is equal to two copies of $\mathscr I_3$, see Example \ref{ex:3pts}. The differential $\partial$ in $CE^\ast((V,h);\R^4)$ of the generators $a$ and $b$ is given by
	\begin{align*}
		\partial a &= e_1 + y^1_{31}bx^0_{12} + y^1_{31}x^0_{13} - y^1_{21}x^0_{12} \\
		\partial b &= x^0_{23} - y^0_ {23} \, ,
	\end{align*}
	and on the generators $x^p_{ij}$ and $y^p_{ij}$ by \eqref{eq:d1}, \eqref{eq:d2} and \eqref{eq:d3}. 

	The Chekanov--Eliashberg dg-algebra of $(K,H)$ is $CE^\ast(\partial L; \partial K_0) = \mathscr I_3$, see Example~\ref{ex:3pts}. Denote the generators of $CE^\ast(\partial L; \partial K_0)$ by $\left\{\widecheck x^p_{ij}\right\}$. The induced dg-algebra map is then
	\begin{align*}
		\varepsilon \colon CE^\ast((V,h);\R^4) &\longrightarrow CE^\ast((K,H); \R \times (\R \times \R^4)) \\
		a &\longmapsto \widecheck x^1_{11} \\
		x^p_{ij}, y^p_{ij} &\longmapsto \widecheck x^p_{ij} \\
		b& \longmapsto 0\, .
	\end{align*}
	The second equation comes from flow lines along the interval of singularities starting at infinity and ending at the minimum $\widecheck x$. To see the first equation, follow the family of holomorphic disks at infinity with boundary on $l_{1}$ into the core disk with boundary $L_{1}$. At some instance the boundary hits the singularity link, projecting to a complex line perpendicular to the complexification of the line of singularities, one finds that this disk is asymptotic to the chord that goes once around. The contributing configuration is then this disk with a flow line in the manifold of once around chords that ends at the minimum $\widecheck x^{1}_{11}$ attached.    
\end{example}

\begin{example}\label{ex:a3_link}
	As in Example \ref{ex:theta_leg} we consider the singularity link of an arboreal singularity in $\R^{4}$. Let
	\[
		L = \{(x_1,x_2)\} \cup \left\{(y_1,x_2) \mid y_1 \geq 0\right\} \cup \left\{(x_1,iy_2) \mid y_2 \leq 0\right\} \subset \R^4\, .
	\]
	Then using $\R^4 \approx T^\ast \R^2$, we have $\Gamma \approx \R^2 \cup L^+_{\left\{x_1 = 0\right\} \cup \left\{x_2 = 0\right\}} \subset T^\ast \R^2$ where $L^+_{\left\{x_1 = 0\right\} \cup \left\{x_2 = 0\right\}}$ is the positive conormal bundle of the hypersurface $\left\{x_1 = 0\right\} \cup \left\{x_2 = 0\right\}\subset \R^2$, where we assume that $\left\{x_1 = 0\right\}$ and $\left\{x_2 = 0\right\}$ have been equipped with opposite co-orientations, and $\Gamma$ is an arboreal singularity. 
	
	Let $\Lambda = \Gamma \cap S^3$ and consider $V = T^\ast \Lambda \subset S^3$ with handle decomposition $h$ as indicated in Figure \ref{f:a3link}. Let $l$ denote the union of the core disks of the top handles. Let $K \subset \R^4$ be a Weinstein domain which agrees with $(-\epsilon, \epsilon) \times V$ in the ideal contact boundary of $\R^4$. Let $H$ be a handle decomposition of $K$ with boundary $h$. Denote the union of the core disks of the critical handles of $H$ by $L$. 
	\begin{figure}[!htb]
		\centering
		\includegraphics{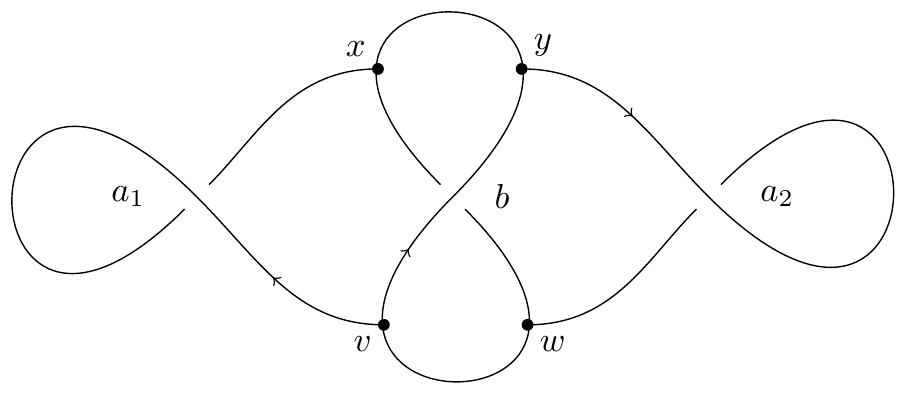}
		\caption{Lagrangian projection of the singularity link $\Lambda$.}
		\label{f:a3link}
	\end{figure}
	The Chekanov--Eliashberg dg-algebra $CE^\ast((V,h);\R^4)$ is generated by Reeb chords of $l \subset S^3$, denoted $a_1$, $a_2$ and $b$ in Figure \ref{f:a3link}. Other generators are Reeb chords of $\partial l \subset \partial V_0$, i.e., of the singularity links at $x,y,z,w$:
	\[
		\left\{x^p_{ij}\right\} \cup \left\{y^p_{ij}\right\} \cup \left\{v^p_{ij}\right\} \cup \left\{w^p_{ij}\right\},
	\]
	where the notation is as usual.
	This collection generates the dg-subalgebra $CE^\ast(\partial l; V_0)$ which contains four copies of $\mathscr I_3$, see Example \ref{ex:3pts}.

	The differential $\partial$ on  $a_1, a_2,b$ is
	\begin{align*}
		\partial a_1 &= e_1 + v^1_{31}bx^0_{12} - v^1_{21}w^0_{23}x^0_{12} + v^1_{31}y^0_{23}x^0_{13} \\
		\partial a_2 &= e_2 - y^1_{31}bw^0_{12}-y^1_{21}x^0_{23}w^0_{12}+y^1_{31}v^0_{23}w^0_{13} \\
		\partial b &= y^0_{23}x^0_{23} - v^0_{23}w^{0}_{23}\, .
	\end{align*}

	As in Example~\ref{ex:death_filling_doesnt_exist}, we show that $\Lambda$ does not admit a singular exact Lagrangian filling $\Gamma$ with only Y-singularities. Assume that such a filling $\Gamma$ exists. Then since $\Lambda$ has four Y-singularities, $\Gamma$ is a union of smooth 2-dimensional strata meeting along $S \times Y$ where $Y$ is a trivalent graph and $S$ is a 1-manifold with at least two components with boundary on the vertices of $\Lambda$. The two components with boundary subdivides the four collection $\left\{x^p_{ij}\right\}$, $\left\{y^p_{ij}\right\}$, $\left\{v^p_{ij}\right\}$, and $\left\{w^p_{ij}\right\}$ into two pairs connected by the components, and $\Gamma$ induces a dg-algebra map
	\[
		\varepsilon \colon CE^\ast((V,h); \R^4) \longrightarrow CE^\ast((K,H);\R \times (\R \times \R^4)) \longrightarrow \mathscr I_3 \ast \mathscr I_3,
	\]
	where $\mathscr{I}_{3}\ast\mathscr{I}_{3}$ denotes the free algebra of two copies of $\mathscr{I}_{3}$ over the ring of idempotents.  
	We find that
	\begin{align}
		\partial(\varepsilon(b)) &= \varepsilon(\partial b) = \varepsilon(y^0_{23})\varepsilon(x^0_{23})-\varepsilon(v^0_{23})\varepsilon(w^0_{23}) \label{eq:b_wrong} \\
		\partial(\varepsilon(a_1)) &= \varepsilon(\partial a_1) = e_1 + \varepsilon(v^1_{31}) \varepsilon(b) \varepsilon(x^0_{12}) - \varepsilon(v^1_{21})\varepsilon(w^0_{23})\varepsilon(x^0_{12}) + \varepsilon(v^1_{31})\varepsilon(y^0_{23})\varepsilon(x^0_{13}) \label{eq:a1_wrong}
	\end{align}
	As in Example \ref{ex:death_filling_doesnt_exist}, the differential on $\mathscr I_3$ changes word length mod 2 and the same is true for the differential on $\mathscr I_3 \ast \mathscr I_3$. We conclude first from \eqref{eq:b_wrong} that $\varepsilon(b)$ is a sum of monomials of generators of odd word length. Hence we have from \eqref{eq:a1_wrong} that $\partial(\varepsilon(a_1)) = e_1 + r_{\text{o}}$, where $r_{\text{o}}$ is a sum of monomials of generators of odd word length. Then as in Example \ref{ex:death_filling_doesnt_exist}, write $\varepsilon(a_1) = t_{\text{o}} + t_{\text{e}}$, with $t_{\rm e}$ and $t_{\rm o}$ linear combinations of monomials of even and odd word length and conclude $\partial t_{\rm o}=e_{1}$. Since no such equation holds in $\mathscr{I}_{3} \ast \mathscr I_3$, the singular exact Lagrangian filling $\Gamma$ with only Y-singularities cannot exist.
\end{example}

\begin{example}
	As in Examples \ref{ex:theta_leg} and \ref{ex:a3_link} we consider the singularity link of an arboreal singularity in $\R^{4}$. Let
	\[
		\Gamma := \R^2 \cup L^+_{\left\{x_1=0\right\}\left\{x_1 \leq 0, \; x_2 = x_1^2\right\}} \subset T^\ast \R^2
	\]
	 where $L^+_{\left\{x_1=0\right\}\left\{x_1 \leq 0, \; x_2 = x_1^2\right\}}$ is the positive conormal bundle of $\left\{x_1=0\right\} \cup \left\{x_1 \leq 0, \; x_2 = x_1^2\right\}\subset \R^2$. Then $\Gamma$ is the arboreal $A_3$-Lagrangian, \cite{Nadler,AGEN1,AGEN3}. 
	The Lagrangian projection of $\Lambda$ is shown in Figure~\ref{f:a3link2}. As above we let $(V,h)$ be a fattening of $\Lambda$ and Let $(K,H) \subset \R^4$ be a filling.
	\begin{figure}[!htb]
		\centering
		\includegraphics{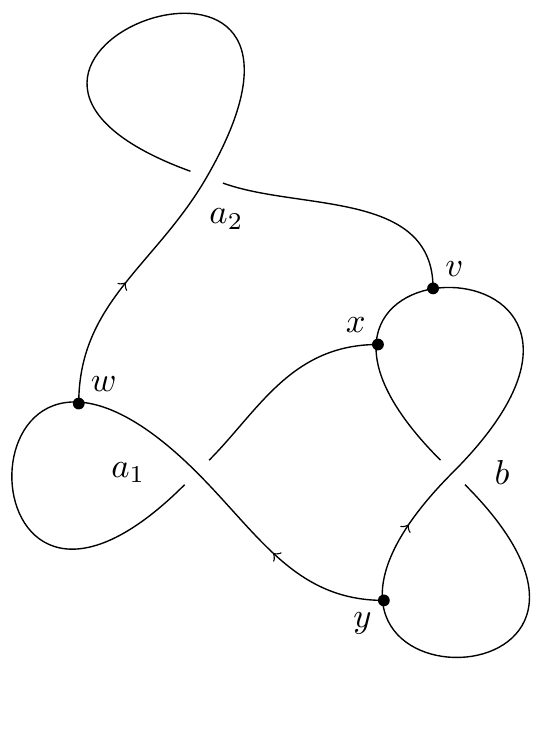}
		\caption{The figure shows the Lagrangian projection of $\Lambda$.}
		\label{f:a3link2}
	\end{figure}
	The Chekanov--Eliashberg dg-algebra $CE^\ast((V,h);\R^4)$ is generated by the long Reeb chords $a_1, a_2, b$ of $l \subset S^3$. The differential $\partial$ is given on the generators $a_1, a_2,b$ by
	\begin{align*}
		\partial a_1 &= w^0_{23} + y^1_{31}bx^0_{12} - y^1_{21}x^0_{12} + y^1_{31}v^0_{23}x^0_{13} \\
		\partial a_2 &= e_3 - w^1_{21}x^1_{31}v^0_{12} - w^1_{31}(y^1_{31}v^0_{13} + (a_1x^1_{31}-y^1_{21}x^1_{32}+y^1_{31}bx^1_{32}+ y^1_{31}v^0_{23}x^1_{33})v^0_{12})\\
		\partial b &= v^0_{23}x^0_{23} - y^0_{23}\, .
	\end{align*}
	The differential on generators $x^p_{ij}$, $y^p_{ij}$, $v^p_{ij}$, $w^p_{ij}$ is as usual. 
	
	As in Example \ref{ex:a3_link}, we show that $\Lambda$ does not admit any singular exact Lagrangian filling $\Gamma$ with only Y-singularities: such $\Gamma$ would induce a dg-algebra map
	\[
		\varepsilon \colon CE^\ast((V,h);\R^4) \longrightarrow CE^\ast((K,H);\R \times (\R \times \R^4)) \longrightarrow \mathscr I_3 \ast \mathscr I_3\, 
	\]
	and we would have
	\[
		\partial(\varepsilon(b)) = \varepsilon(\partial b) = \varepsilon(v^0_{23})\varepsilon(x^0_{23}) - \varepsilon(y^0_{23})\, .
	\]
	Recall from Example \ref{ex:a3_link} that the differential in $\mathscr I_3 \ast \mathscr I_3$ changes word length mod 2. We write $\varepsilon(b) = t_{\text{o}} + t_{\text{e}}$, where $t_{\rm e}$ and $t_{\rm o}$ are linear combinations of monomials of even and odd word length, respectively. Then $\partial(t_{\text{e}}) = \varepsilon(y_{23}^0)$, but this is not true in $\mathscr I_3 \ast \mathscr I_3$, since $\varepsilon(y^0_{23})$ equals a short chord in $CE^\ast(\partial L; \partial K_0)$ and such a chord is not homologous to $0$ in $\mathscr{I}_{3}$.
	
\end{example}
\begin{example}\label{ex:three_spheres}
	Consider $\R^6$ with its standard symplectic form and ideal contact boundary standard contact $S^5$. As in the $\R^{4}$-examples above, we draw Legendrians in a Darboux chart of $S^{5}$ that we think of as standard contact $\R^{5}$. We will consider a singular Legendrian $\Lambda\subset S^{5}$ which has one Reeb chord of length zero. Resolving this double point we get embedded Legendrian tori. We show here the dg-algebra of the singular Legendrian $\Lambda$ when equipped with suitable augmentations of its singularity link subalgebra is dg-algebra equivalent to any nearby smooth Legendrian torus, compare Corollary \ref{cor:CEconnectsumtrivial}. The smooth nearby Legendrian tori were studied in \cite{DR}. 
	
	Consider the singular Legendrian submanifold $\Lambda$ with front as in Figure~\ref{fig:singulartorus}. The singularity of $\Lambda$ is one immersed point (which can be though of as a Reeb chord of length zero). Let $V = T^\ast \Lambda \subset S^5$ be a Weinstein thickening of $\Lambda$ with handle decomposition $h$. Here $h$ has one $0$-handle centered at the singular double points of $\Lambda$ and one $1$-handle which together form $S^{1}\times D^{3}$. Finally there is one $2$-handle. We denote its core disk $l$. The attaching sphere $\partial l$ for $l$ can then be described as follows. The intersection of the boundary of the $0$-handles with $\Lambda$ is a Legendrian Hopf link in $S^{3}$. The $1$-handle is attached on this Hopf link, connecting its components by the standard two strand Legendrian through the handle, see \cite{ENg}. 
	\begin{figure}[!htb]
		\centering
		\includegraphics{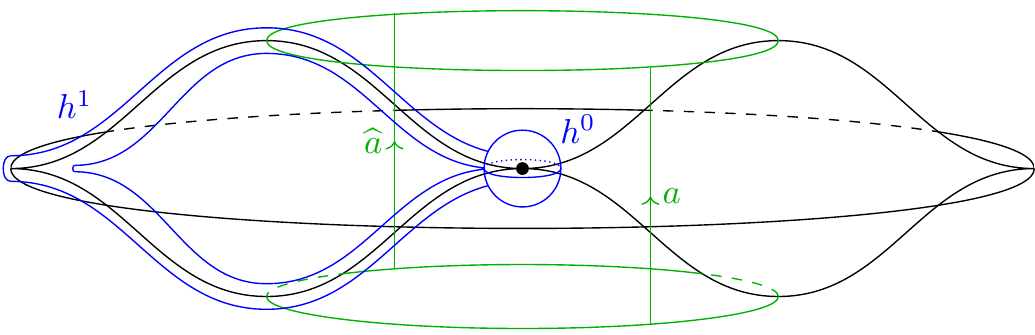}
		\caption{The front of the singular Legendrian $\Lambda$ with an immersed point. The $0$-handle is a neighborhood of the immersed point and the $1$-handle is a neighborhood of a curve connecting the singular point to itself. The family of external Reeb chords is indicated in green.}
		\label{fig:singulartorus}
	\end{figure}
	
	The dg-algebra of $(V,h)$ is then the following. At the minimum of the $1$-handle sits the dg-algebra $\mathscr{I}_{2}$, we denote its generators $c_{ij}^{p}$. The differential is as in Example \ref{ex:3pts}. We represent the singularity link as the boundary of two transverse planes. This dg-algebra is quasi-isomorphic to the standard Hopf link dg-algebra of \cite[Section 3]{DRET}. Reeb chords of the Hopf link in $S^{3}$ come in pairs of $S^{1}$-families of length $\tfrac{k\pi}{2}$, $k>0$. We will use only the shortest chord families which after Morsification give rise to two chords each: $p$ and $\widehat p$, and $q$ and $\widehat q$. Since the 1-handle connects the Hopf link components the differential is as follows, see Figure \ref{fig:differentialsingtor}:
	\[ 
	\partial p =\partial q =0,\quad \partial\widehat p = p-c_{21}^{1}pc_{12}^{0},\quad \partial \widehat q=q-c_{12}^{0}qc_{21}^{1}.
	\] 
	Also the exterior Reeb chords come in an $S^{1}$-family. After Morsification we get two chords $\widehat a$ and $a$. The differential is as follows:
	\begin{align*} 
	\partial\widehat a &= a - c_{21}^{1}ac_{12}^{0} + \widehat p - c_{11}^{1},\\ 
	\partial a &= e-p,
	\end{align*}
	see Figure \ref{fig:differentialsingtor}.
	
	\begin{figure}[!htb]
		\centering
		\includegraphics{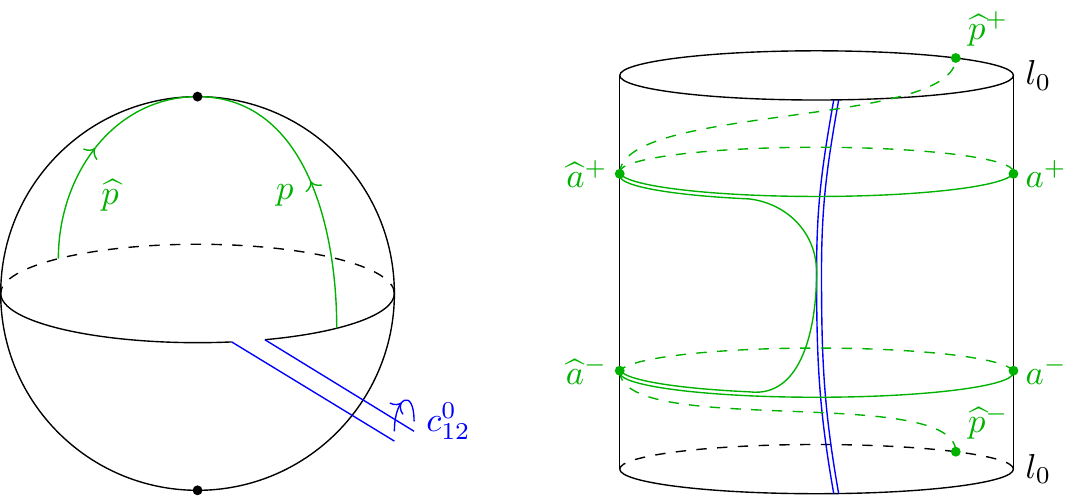}
		\caption{The differential.}
		\label{fig:differentialsingtor}
	\end{figure}
	The left hand picture shows the differential of $\widehat p$: both flow lines of Reeb chord endpoints hit the attaching locus of $h^{1}$, enter the handle and hit the short chords in the middle of the handle. 
	
	The right hand picture shows the boundary of the curves that contribute to the differential of the exterior chords on a cylinder parameterizing the singular Legendrian, the boundary maps to the core disk of the $0$-handle, $l_0$. The differential of $a$ is the vertical line connecting $a^{+}$ to $a^{-}$ and the vertical lines connecting to the boundary. The disks for $\partial\widehat{a}$ are the line which is tangent to the attaching locus of the $1$-handle. It gives $c^{1}_{11}$, a flow line to the minimum in the $1$-handle connects the tangency point to $c^{1}_{11}$. The pair lines intersecting the $1$-handle and ending at $a^{\pm}$ gives $c_{21}^{1}ac_{12}^{0}$, and the pair ending at $a^{\pm}$ not intersecting the $1$-handle gives $a$, flow lines to the short chords are split off at the intersections. Finally the lines to $\widehat p^{\pm}$ gives the contribution $\widehat p$.
	
	The Legendrian tori in \cite{DR} are obtained by resolving the double point of $\Lambda$ in two different ways, as a Lagrangian cone and as a cusp edge. This in turn correspond to distinct Lagrangian fillings of the Hopf link which induces different augmentations of its dg-algebra, see \cite[Equations (4.2) and (4.4)]{DRET}. The resulting augmentations here are $\epsilon$ and $\epsilon'$:
	\begin{align*}
	&\epsilon(c_{12}^{0})=\epsilon'(c_{12}^{0})=\lambda,\quad \epsilon(c_{21}^{1})=\epsilon'(c_{21}^{1})=\lambda^{-1},\\
	&\epsilon(p)=\mu,\quad\epsilon'(p)=\mu-\mu\lambda.
	\end{align*} 
\end{example}

\bibliographystyle{hplain}
\bibliography{cesingrefs}
\end{document}